\documentclass[12pt, reqno]{amsart}
\usepackage{amsmath}
\usepackage{amssymb, environ}
\usepackage[margin=1.2in]{geometry}
\usepackage{lineno,hyperref}
\usepackage{url}
\usepackage[title,toc,titletoc,header]{appendix}

\usepackage [english]{babel}
\usepackage [autostyle, english = american]{csquotes}
\MakeOuterQuote{"}
\newtheorem{thm}{Theorem}[section]

\newtheorem{lemma}[thm]{Lemma}
\newtheorem{cor}[thm]{Corollary}

\newtheorem{propose}[thm]{Proposition}

\newtheorem{notations}[thm]{Notations}

\newtheorem{rmk}[thm]{Remark}
\newtheorem{defn}[thm]{Definition}
\newtheorem{ex}[thm]{Example}
\numberwithin{equation}{section}
            
\newcommand{\longto}{\longrightarrow}

\newcommand{\sF}{{\mathcal F}} 
\newcommand{\sL}{{\mathcal L}} 
\newcommand{\sM}{{\mathcal M}}
\newcommand{\sN}{{\mathcal N}} 
\newcommand{\sO}{{\mathcal O}}
\newcommand{\sP}{{\mathcal P}}

\newcommand{\F}{{\mathbb F}}
\newcommand{\M}{{\mathbb M}}
\newcommand{\N}{{\mathbb N}}
\renewcommand{\P}{{\mathbb P}} 
\newcommand{\Q}{{\mathbb Q}}
\newcommand{\R}{{\mathbb R}} 
\newcommand{\Z}{{\mathbb Z}}

\begin{document}
	\title{$\beta$-density function on the class group of projective toric varieties}
	
	\author{Mandira Mondal}
	\address{M. Mondal, 
		Chennai Mathematical Institute, H1, SIPCOT IT Park,
		Siruseri, Kelambakkam 603103, India}
	\email{mandiram@cmi.ac.in}
	
	\date{\today}

	\thanks{The author is supported by NBHM postdoctoral fellowship, India and partially by an Infosys Foundation
		fellowship.}
	
	\begin{abstract}
		We prove the existence of a compactly supported, continuous (except at finitely many points)
		function $g_{I, {\bf m}}: [0, \infty)\longto \R$ for all monomial prime ideals $I$ of $R$ of height one
		where $(R, {\bf m})$ is the homogeneous coordinate ring associated to a projectively normal
		toric pair $(X, D)$, such that
		$$\int_{0}^{\infty}g_{I, {\bf m}}(\lambda)d\lambda=\beta(I, {\bf m}),$$
		where $\beta(I, {\bf m})$ is the second coefficient of the Hilbert-Kunz function of $I$ with respect to
		the maximal ideal ${\bf m}$, as proved by Huneke-McDermott-Monsky \cite{HMM04}. Using the above
		result, for standard graded normal affine monoid rings we give a complete
		description of the class map $\tau_{{\bf m}}:\text{Cl}(R)\longto \R$ introduced in \cite{HMM04} to prove
		the existence of the second coefficient of the Hilbert-Kunz function. Moreover, we show the function $g_{I, {\bf m}}$
		is multiplicative on Segre products with the expression involving the first two coefficients of the Hilbert plolynomial
		of the rings and the ideals.
	\end{abstract}
	
	\maketitle
	
	\section{Introduction}
	\label{intro}
	Let $R$ be a Noetherian ring of prime characteristic $p > 0$ and of dimension $d$ and let $\eta\subseteq R$ be an ideal of finite colength. Let $M$ be a finitely generated $R$-module. The Hilbert-Kunz function of $M$ with respect to the ideal $\eta$ is defined as
	$$\mbox{HK}(M, \eta)(n):=\ell(M/\eta^{[q]}M)$$
	where $q=p^n$, the ideal $\eta^{[q]} =n$-th Frobenius power of the ideal $\eta$ and $\ell(M/\eta^{[q]}M)$ denotes the length of the $R$-module $M/\eta^{[q]}M$. The limit $$ \lim_{n\to\infty} \frac{1}{q^{d}}\ell(M/\eta^{[q]}M)=:e_{HK}(M,\eta)$$
	exists \cite{Mon83} and is called the Hilbert-Kunz multiplicity of $M$ with respect to the ideal $\eta$. In addition to the above conditions, when $R$ is an excellent normal domain, Huneke, McDermott and Monsky \cite[Theorem 1]{HMM04} have shown the existence of a real number $\beta(M, \eta)$ such that 
	$$\mbox{HK}(M, \eta)(n)=e_{HK}(M,\eta)q^d+\beta(M, \eta)q^{d-1}+O(q^{d-2}).$$
	In the course of the proof of the above result, they have asserted the existence of a homomorphism
	$\tau_\eta:\text{Cl}(R)\longto \R$ on the class group of $R$, $\text{Cl}(R)$, the quotient of the free abelian group on the
	height one prime ideals of $R$ by the subgroup of principal divisors. 
	Let $M$ be a finitely generated $R$-module. Then $M$ admits a finite filtration $0\rightarrow\cdots M_{i-1}\rightarrow M_{i}\cdots \rightarrow M$ such that $M_i/M_{i-1}$ is isomorphic to $R/P_i$ with $P_i$ prime ideals in $R$. Consider the divisor $-\sum P_i$ where the sum is taken over all $P_i$
	appearing in the quotients $M_i/M_{i-1}$ that are of height one. The image of this divisor in the class group of $R$ is independent of the filtration chosen for $M$, and is defined as the class of $M$, denoted by $c(M)$.
	Let $M$ be a finitely generated torsion-free $R$-module. By \cite[Corollary 1.10]{HMM04}, the limit
	$$\tau_{\eta}(M):=\lim_{q\to \infty}\frac{1}{q^{d-1}}\left[\ell(M/\eta^{[q]}M)-\text{rank}(M)\ell(R/\eta^{[q]})\right]$$
	is well defined and depends only on $c(M)$, the class of $M$ in $\text{Cl}(R)$. When $R$ is $F$-finite,
	$$\beta(M, \eta)=\tau_\eta(c(M))-\frac{\text{rank}(M)}{p^d-p^{d-1}}\tau_\eta(c(^{1}R)),$$
	where $^{1}R$ denotes the finitely generated module $R$ over itself with the action given by the first Frobenius homomorphism. 
	
	The result of Huneke--McDermott--Monsky was generalised  by 
	Hochster-Yao in \cite{HY09} from normal rings to the equidimensional reduced rings such that the singular locus is given by an ideal of
	height at least $2$.
	Chan and Kurano have proved the result for reduced rings regular in codimension one \cite{CK16}. 
	For a normal affine monoid $R$, Bruns in \cite{Bru05} have proved that
	HK function is a quasi polynomial and gave another proof of the 
	existence of the constant second coefficient $\beta(R, {\bf m})$.

	In order to study $e_{HK}(M, \eta)$,  
	when  $R$ is a standard graded ring ($\dim(R)\geq 2)$ with a homogeneous ideal $\eta$ of finite colength
	and $M$ is a finitely generated non-negatively graded $R$-module, Trivedi has defined the notion of 
	Hilbert-Kunz density function, and obtained its relation with the 
	\textnormal{\mbox{HK}} multiplicity \cite[Theorem 1.1]{Tri18}: 
	{\it The sequence of functions $\{f_n(M,\eta):[0,\infty)\longto 
		\R_{\geq 0}\}_n$
		given by 
		$$f_n(M, \eta)(\lambda) = \frac{1}{q^{d-1}}\ell(M/\eta^{[q]}M)_{\lfloor q\lambda\rfloor}$$
		converges uniformly to a compactly supported continuous function 
		$f_{M,\eta}:[0,\infty) \to 
		\R_{\geq 0}$, such that 
		$$e_{HK}(M, \eta)=\int_0^\infty f_{M, \eta}(\lambda)d\lambda.$$
	}
	
	\noindent We call $f_{M, \eta}$ the Hilbert-Kunz density function or the $\mbox{HK}$ density function of $M$
	with respect to the ideal $\eta$.
	The existence of a  uniformly converging sequence  makes 
	the density function a more refined and useful invariant (compared to $e_{HK}$) 
	in the graded situation (\cite{Tri17}, \cite{Tri19}, \cite{TW20}). Applying  the theory of \mbox{HK} density functions  to  projective toric varieties 
	(denoted here as toric pairs $(X,D)$), one obtains 
	\cite[Theorem~6.3]{MT19} an algebraic  characterization of the tiling property of 
	the associated polytopes $P_D$ (in the ambient lattice) 
	in terms of the asymptotic growth of $e_{HK}$, i.e., $e_{HK}(R, {\bf m}^k)$ relative to $e_{0}(R, {\bf m}^k)$ (the Hilbert Samuel multiplicity of $R$ with respect to the ideal ${\bf m}^k$) as $k\to \infty$.

	Let $(X, D)$ be a toric pair, i.e., $X$ is a projective toric variety over an algebraically closed field of characteristic $p>0$,
	with a very ample $T$-Cartier divisor $D$ and let $R$ be the homogeneous coordinate ring of
	$X$, with respect to the embedding given by the very ample
	line bundle $\mathcal{O}_X(D)$, with homogeneous maximal ideal $\textbf{m}$.
	There is a convex 
	lattice polytope 
	$P_D$ as in (\ref{eq2.1}), a convex polyhedral cone $C_D$  and a bounded body 
	$\sP_D$ as in (\ref{eq2.3}), associated to a toric pair $(X, D)$. Such a bounded body was introduced by 
	K. Eto (see \cite{Eto02}), in order to study the HK multiplicity for a toric ring, and  
	he proved that $e_{HK}$ is the relative volume of such a body (we use the 
	notation $\textnormal{\mbox{rVol}}_{n}$ to denote the ${n}$-dimensional 
	relative volume function). 
	In \cite{MT19}, it was  shown that the HK density function at $\lambda $ is 
	the relative volume of the $\{z=\lambda\}$ slice of $\sP_D$. 
	
	Similar to the $\mbox{HK}$ density function, for a `projectively normal' toric pair $(X, D)$ (i.e., $(X, D)$ is a toric pair such that the coordinate ring $R$ is an integrally closed domain), it was shown in \cite{MT20} that
	there exists a {\it$\beta$-density function}
	$g_{R, {\bf m}}: [0,\infty)\to \R$
	which similarly refines the $\beta$-invariant of \cite{HMM04}. 
	More precisely, it was shown that {\it the sequence of functions $\{g_n(R, {\bf m}):[0,\infty)\longto\R\}_{n\in\N}$, given by
		\begin{equation}\label{*g}
			g_n(R, {\bf m})(\lambda)=\frac{1}{q^{d-2}}\left(
			\ell(R/\textbf{m}^{[q]})_{\lfloor q\lambda \rfloor}
			-{f_{R, {\bf m}}}(\lfloor q\lambda\rfloor/q) q^{d-1}\right),\end{equation}
		converges uniformly to a compactly supported continuous (except possibly on a finite set)  function $g_{R, {\bf m}}$ such that
		$\int_{0}^{\infty}g_{R, {\bf m}}(x)dx=\beta(R, {\bf m})$}. It was shown that the $\beta$-density function $g_{R, {\bf m}}$ at $\lambda$ is expressible in terms of the 
	relative volume of the $\{z=\lambda\}$ slice of the boundary, $\partial(\sP_D)$, 
	of $\sP_D$ (stated in this paper as Theorem \ref{thbeta}).
	
	In regard to Theorem \ref{thbeta}, one would like to ask whether there exists the notion of $\beta$-density function (with respect to the homogeneous maximal ideal ${\bf m}$)
	for all finitely generated non-negatively graded $R$-modules $M$ which refines the invariant $\beta(M, {\bf m})$. In this paper we answer this question affirmatively for monomial prime ideals of $R$ of height one. Using this result, we define a `$\tau$-density function'  $\alpha_{I, \bf m}:[0, \infty)\longto \R$ for these ideals which describe the value of the function $\tau_{\bf m}:[0, \infty)\longto \R$ for these ideals via a simple integral formula, i.e., $\int_{0}^{\infty}\alpha_{I, {\bf m}}(x)dx=\tau_{\bf m}(I)$. This gives a complete description of the homomorphism $\tau=\tau_{\bf m}$ since the class group of $R$ is generated by its monomial prime ideals of height one. 
	
	Let $I=p_{F}$ be a monomial prime ideal of height one, associated to a facet $F$ of $P_D$. 
	To prove the existence of the $\beta$-density function for $I$ with respect to the homogeneous maximal ideal ${\bf m}$, consider
	the sequence of functions $\{g_n{(I, {\bf m})}:[0, \infty)\longto\R\}_n$, given by 
	$$g_n{(I, {\bf m})}(\lambda)=\frac{1}{q^{d-2}}\left(\ell(I/{\bf m}^{[q]}I)_{\lfloor\lambda q\rfloor}
	-f_{I, {\bf m}}(\lfloor\lambda q\rfloor/q)q^{d-1}\right).$$
	Let $\sigma_F:\R^d\longto\R$ be the support function for the facet of $C_D$ corresponding to the facet $F$ of $P_D$ and let $H_{F, \mu}=\{x\in\R^d\mid \sigma_F(x)=\mu\}$ for all $\mu\in \Q_{\geq 0}$. Also, let $\mu_{D, F}=\{\mu\in \Q_{>0}\mid u\in H_{F, \mu}\text{ for some }u\in P_D\cap \sM\}$, where $\sM$ is the ambient lattice associated to the torus $T\subset X$ (see Section 2).
	We prove the following main result. 
	\begin{thm}\label{thidealbeta} Let $(X, D)$ be a projectively normal toric pair of dimension $\geq 2$ and let $(R, {\bf m})$ be the associated homogeneous coordinate ring. Let $I=p_F$ be a monomial prime ideal of height one, associated to a facet $F$ of the polytope $P_D$. 
		There exists a finite set $v_{\sP_D, F}\subset [0, \infty)$ such that for any compact set $V\subset [0, \infty)\setminus v_{\sP_D, F}$, the sequence of functions $\{g_n{(I, {\bf m})}|_V\}$
		converges uniformly to a function $g_{I, {\bf m}}|_V$ where $g_{I, {\bf m}}: [0, \infty)\setminus v_{\sP_D, F}\longto \R$ is a compactly supported continuous function given by
		$$g_{I, {\bf m}}(\lambda)=g_{R, {\bf m}}(\lambda)-f_{R/I, {\bf m}/I}(\lambda)+\sum_{\mu\in \mu_{D,F}}\textnormal{\mbox{rVol}}_{d-2}(\partial{(\sP_D)}\cap H_{F, \mu}\cap \{z=\lambda\}).$$
		Here $f_{R/I, {\bf m}/I}: [0, \infty) \longto \R_{\geq 0}$ is the \textnormal{\mbox{HK}} density function of the graded ring $R/I$ with respect to the homogeneous maximal ideal ${\bf m}/I$.
		
		Moreover,
		$$\beta(I, {\bf m})=\int_{0}^{\infty}g_{I, {\bf m}}(\lambda)d\lambda.$$ 
	\end{thm}
	
	Now we give a brief sketch of the proof of Theorem \ref{thidealbeta}. Since $f_{R, {\bf m}}(\lambda)=f_{I, {\bf m}}(\lambda)$ for all $\lambda\in[0,\infty)$ (\cite[Proposition 2.14]{MT19}), we note that
	\begin{equation}\label{eq01}g_n(I, {\bf m})(\lambda)=g_n(R, {\bf m})(\lambda)+f_n(R/I, {\bf m}/I)(\lambda)+\psi_n(\lambda),
	\end{equation}
	where
	the function $\psi_n:[0, \infty)\longto \R$ is given by
	$$\psi_n(\lambda)=\frac{1}{q^{d-2}}\ell\left(\frac{m^{[q]}\cap I}{m^{[q]}I}\right)_{\lfloor q\lambda\rfloor}.$$
	Thus we need to show the sequence of functions $\{\psi_n\}$ converges uniformly. We note that the proof of 
	`existence' of an invariant or a property in Hilbert-Kunz theory, often boils down to bounding the 
	`correction' term in a converging sequence. 
	For example, for the proof of \cite[Theorem 1]{HMM04}, for any torsion-free
	$R$-module $M$ with $c(M)=0$, they show that $\ell(M/\eta^{[q]}M)-\textnormal{rank}(M)\ell(R/\eta^{[q]})=O(q^{d-2})$. The \cite[Lemma 1.2]{HMM04} is crucial for this proof which uses a similar order bound on the
	length $\ell(T/\eta^{[q]}T)=O(q^{\dim(T)})$ for any finitely generated $R$-module $T$, due to Monsky \cite{Mon83}. To prove the 
	existence of the Hilbert-Kunz density function, in \cite[Proposition 2.12]{Tri18} Trivedi shows $|f_n(M, \eta)(\lambda)-f_{n'}(M, \eta)(\lambda)|=O(1/q)$ for all $n'>n\gg 0$. In \cite{MT20}, for $\lambda\in[0,\infty)$ and $\lambda_n:={\lfloor q\lambda\rfloor}/q \notin v(\sP_D)$ (Notations \ref{ndn}(2)), it is shown that $g_{n}(R, {\bf m})(\lambda)=g(R, {\bf m})(\lambda_n)+c(\lambda_n)/q $ with $|c(\lambda_n)|<\tilde{C}$, a constant independent of $\lambda$ and $n$. 
	
	In this paper, we use a similar approach to bound the error term in the converging sequence of functions $\{\psi_n\}$. In particular, we show that (Lemma \ref{lemconv}) there exists a finite set $v_{\sP_D, F}\subset [0, \infty)$ such that for all $\lambda\in [0, \infty)$ and for all $n\in \N$ with $\lambda_n\notin v_{\sP_D,F}$,
	$$\psi_n(\lambda)=\sum_{\mu\in \mu_{D,F}}\textnormal{\mbox{rVol}}_{d-2}(\partial{(\sP_D)}\cap H_{F, \mu}\cap \{z=\lambda_n\})+\frac{c_\lambda(n)}{q}$$
	where $|c_\lambda(n)|\leq C$ for some constant $C$, independent of $\lambda$ and $n$.
	Hence for any compact set $V\subset [0, \infty)\setminus v_{\sP_D,F}$, the sequence of functions $\{\psi_n|_V\}$ converges uniformly  
	to the function $\Psi_F|_V$, given by
	$\lambda\mapsto \sum_{\mu\in \mu_{D,F}}\textnormal{\mbox{rVol}}_{d-2}(\partial{(\sP_D)}\cap H_{F, \mu}\cap \{z=\lambda\}).$
	This observation along with Equation (\ref{eq01}), Theorem \ref{thbeta} and the property of \mbox{HK} density function give us the proof of the first part of the main Theorem. 
	
	Now, since $f_{R, {\bf m}}(\lambda)=f_{I, {\bf m}}(\lambda)$ for all $\lambda\in[0,\infty)$,
	by \cite[Lemma 40]{MT20} we have
	$$\int_{0}^{\infty} f_{I, {\bf m}}(\lfloor\lambda q\rfloor/q) d\lambda=e_{HK}(I, {\bf m}) +O(1/q^{d-2}).$$
	Now, a similar approximation of the integral of the function $g_{I, {\bf m}}$ by the integral of the functions $g_n(I , {\bf m})$, as was approximated the integral of the function $g_{R, {\bf m}}$ by the integral of the functions $g_n(R , {\bf m})$ in \cite{MT20}, gives us that
	$\int_{0}^{\infty}g_{I, {\bf m}}=\beta(I, {\bf m})$. 
	
	\vspace{.2cm}

	\noindent {\bf Acknowledgement}: I would like to express my gratitude to Prof. V. Trivedi for her continuous encouragement and insightful discussions.
	
	\section{Density functions on projective toric varieties}\label{Sec:2}
	In this paper we work over an algebraically closed field $K$ with char $p > 0$.
	Let  $\sN$ be  a lattice (which is isomorphic to $\Z^{d-1}$)
	and let  $\sM = \text{Hom}(\sN, \Z)$ denote the dual lattice 
	with a dual pairing $\langle\ , \rangle$.  
	Let $T = \text{\mbox{Spec}}(K[\sM])$ be the torus with character lattice $\sM$ and let $X$ be a complete toric variety over $K$ with fan $\Delta \subset \sN\otimes\R:=\sN_\R$. The irreducible subvarieties of 
	codimension $1$ of $X$ which are stable under the action of the torus $T$  correspond to the edges (one dimensional cones)
	of $\Delta$. If  
	$\tau_1,\ldots, \tau_n$ denote the edges of the fan $\Delta$, then these 
	divisors are the orbit closures $D_i= V(\tau_i)$. Let $v_i$ be the first lattice
	point along the edge $\tau_i$. A very ample
	$T$-Cartier divisor 
	$D=\sum_i a_iD_i$ ($a_i\in \Z$) determines a  
	convex lattice polytope in $\sM_{\mathbb{R}}:=\sM\otimes\R$ defined by
	\begin{equation}\label{eq2.1}
		P_D=\{u\in \sM_{\mathbb{R}} \ | \ \langle u, v_i\rangle\geq -a_i 
		~~\text{for all}\ 
		i\ \}\end{equation}
	and the induced embedding of $X$ in $\P^{l-1}$ is given by 
	$$\phi=\phi_D: X\to\mathbb{P}^{l-1},\ 
	\ x\mapsto ({\chi}^{u_1}(x):\cdots: {\chi}^{u_l}(x)),$$ 
	where $P_D\cap \sM=\{u_1, u_2,\ldots, u_l\}$ (for more detailed discussion,
	see  \cite{Ful93}). 
	
	The ring $K[{\chi}^{(u_1,1)},\ldots, {\chi}^{(u_l,1)}]$ is the homogeneous coordinate ring of $X$ with respect 
	to this embedding. 
	We have an isomorphism of graded rings \cite[Proposition 1.1.9]{CLS11}
	
	\begin{equation}\label{**} \frac{K[Y_1,\ldots, Y_l]}{I}\simeq K[{\chi}^{(u_1,1)},
		\ldots, {\chi}^{(u_l,1)}] =:R,\end{equation}
	where, the kernel $I$ is generated by the binomials  of the form
	$$Y_1^{a_1}Y_2^{a_2}\cdots Y_l^{a_l}-Y_1^{b_1}Y_2^{b_2}\cdots Y_l^{b_l}$$
	where $a_1,\ldots, a_l, b_1,\ldots, b_l$ are nonnegative integers satisfying 
	the equations 
	$$a_1u_1+\cdots+a_lu_l=b_1u_1+\cdots+b_lu_l\ \ \text{and}\ \ 
	a_1+\cdots+a_l=b_1+\cdots+b_l.$$
	Due to this isomorphism, we can consider $R= K[S]$ as a standard 
	graded ring with $\deg(\chi^{(u_i, 1)} )= 1$, where $S$ is the semigroup generated by $\langle (P_D\cap \sM)\times \{z=1\}\rangle$
	in  $\R^d$.
	
	Let $C_D$ be the cone generated by $\langle (P_D\cap \sM)\times \{z=1\}\rangle$
	in  $\R^d$. The prime ideals of the polytopal ring $R$ is in one-to-one correspondence with
	faces of $C_D$, given by $$C_F\leftrightarrow p_F:=\textnormal{ ideal of }R\text{ generated by the set of monomials } \{\chi^{\bf \nu}\mid{\bf \nu}\in S\setminus C_F\}\subset R$$ where $C_F$ is the face 
	of $C_D$ corresponding to a face $F$ of $P_D$ \cite[Proposition 2.36, Proposition 4.32]{BG09}. The height one prime ideals correspond to the facets of $P_D$ under this correspondence \cite[Proposition 4.35]{BG09}.
	In this case, the valuation $v_{p_F}$ is the unique extension of the support form $\sigma_F$ of $C_D$ associated with the facet $C_F$. When $(X, D)$ is a projectively normal toric pair, i.e., the associated homogeneous coordinate ring $R$ is an integrally closed domain, the semigroup
	$S=C_D\cap\Z^d$ and the divisorial monomial ideals of $R$ are exactly the $R$-submodules of $R=K[S]$ whose monomial
	basis is determined by a system
	$$\{x\in \R^d\mid \sigma_F(x)\geq n_F, F\text{ is a facet of }P_D\}$$
	for $n_F\in \Z$ \cite[Theorem 4.53]{BG09}. Let $\text{Div}(S)$ denote the subgroup of $\text{Div}(R)$ generated by monomial divisorial prime ideals and let $\text{Princ}(S)$ be its subgroup generated by principal monomial ideals. The class group of the semigroup $S$, denoted $\text{Cl}(S)=\text{Div}(S)/\text{Princ}(S)$ is generated by the classes of the ideals $p_{F}$ where $F$ runs over the set of facets of $P_D$ \cite[Corollary 4.55]{BG09} and is isomorphic to the group Cl$(R)$, the class group of $R$ \cite[Theorem 4.59]{BG09}.
	
	For a toric pair $(X, D)$, let 
	\begin{equation}\label{eq2.3}\mathcal{P_D}=\{p\in C_D\ |\ p\notin (u,1) +C_D,
		\text{for every}~u\in P_D\cap \sM\}.
	\end{equation}
	By result of Eto we have $e_{HK}(R, {\bf m})=\mbox{Vol}_d(\sP_D)=\mbox{Vol}_d(\overline{\sP}_D)$ \cite[Theorem 2]{Eto02}.
	Here $\mbox{Vol}_n$ denotes the $n$-dimensional volume. Moreover,
	$$\mbox{HKd}(R, {\bf m})(\lambda)=\mbox{Vol}_{d-1}(\sP_D\cap\{z=\lambda\})=\mbox{Vol}_{d-1}(\overline{\sP}_D\cap\{z=\lambda\})$$
	for all $\lambda\in [0, \infty)$ \cite[Theorem 1.1]{MT19}. In particular, it is a piecewise polynomial function.
	
	We recall the following result from \cite{MT20}:
	\begin{thm}\cite[Theorem 2, Corollary 3]{MT20}\label{thbeta} Let $(R, {\bf m})$ be the  
		homogeneous coordinate ring of dimension $d\geq 3$, associated 
		to the projectively normal toric pair $(X, D)$.
		Then  there exists a finite set $v(\sP_D)\subseteq 
		\R_{\geq 0}$ such that, for any compact set 
		$V\subseteq \R_{\geq 0}\setminus v(\sP_D)$, the sequence $\{g_n|_V\}_n$ $($as described in $(\ref{*g})$$)$
		converges uniformly  to $g_{R, {\bf m}}|_V$, where 
		$g_{R,{\bf m}}:\R_{\geq 0}\setminus v(\sP_D)\longto \R$ is a continuous 
		function  given by 
		$$g_{R, {\bf m}}(\lambda) = 
		\textnormal{\mbox{rVol}}_{d-2}\left(\partial({\sP_D})\cap \partial(C_D)
		\cap \{z=\lambda\}\right)
		- \frac{\textnormal{\mbox{rVol}}_{d-2}\left(\partial({\sP_D})\cap 
			\{z=\lambda\}\right)}{2}.$$
		
		Moreover, we have 
		$$\beta(R, {\bf m}) = \int_0^\infty g_{R, {\bf m}}(\lambda)d\lambda = 
		\textnormal{\mbox{rVol}}_{d-1}\left(\partial({\sP_D})\cap \partial(C_D)
		\right)
		- \frac{\textnormal{\mbox{rVol}}_{d-1}\left(\partial({\sP_D})\right)}{2}.$$
	\end{thm}
	
	Throughout the paper, we use the following notations.
	\begin{notations}\rm\label{n1}
		\begin{enumerate}
			\item For a facet $F$ of $P_D$, let $C_F$ be the corresponding facet of $C_D$ with
			supporting hyperplane $H_F$ and support form $\sigma_F:\R^d\longto \R$. Hence 
			$H_F=\{x\in\R^d\mid \sigma_F(x)=0\}$ and $C_F=C_D\cap H_F$. Note that
			$$C_D=\cap_{\{F\mid F\text{ is a facet of }P_D\}}\{x\in\R^d\mid\sigma_F(x)\geq 0\}.$$
			
			\item $C_u=(u, 1)+C_D$ for $u\in P_D\cap \sM$. 
			
			\item For the ideal $I=p_F$, we set
			$$C_{I}=\{x\in C_D\mid \sigma_F(x)>0 \}.$$

			\item For a set $A\subset \sM_\mathbb{R}\times \R\simeq \mathbb{R}^{d}$, we denote 
			$$A\cap \{z=\lambda\} := A\cap \{({\bf x}, \lambda)\mid {\bf x}\in \R^{d-1}\}.$$ 
			
			\item For a bounded set $A\subset \mathbb{R}^{d}$, we set $L(A) = A\cap (\sM\times\Z) =$ the (finite) 
			set of lattice points of $A$. 
			
			\item For $m\in\Z$, let us denote the set of lattice points in the hyperplane $\{z=m\}$ by $\Lambda_m$, i.e.,
			$\Lambda_m=\Z^{d-1}\times\{z=m\}\subseteq \R^d$.
			
		\end{enumerate}
		
	\end{notations}

	\begin{lemma}\label{lemma4.1} Let $C_F\subset C_D$ be the cone generated by the facet $F$ of $P_D$. Then
		\begin{enumerate}
			\item for $q\in \N$, we have
			$C_{F}\setminus \cup_{u\in L(P_D)}q(u, 1)+C_D=C_{F}\setminus \cup_{u\in L(F)}q(u, 1)+C_F$.
			\item $[C_{D}\setminus \cup_{u\in L(P_D)}C_u]\cap C_F=C_{F}\setminus \cup_{u\in L(P_D)}C_u=C_{F}\setminus \cup_{u\in L(F)}(u, 1)+C_F.$
		\end{enumerate}
	\end{lemma}
	
	\begin{proof}
		Proof of Part $(1)$: Since $\cup_{u\in L(P_D)}q(u, 1)+C_D\supset\cup_{u\in L(F)}q(u, 1)+C_F$, it is enough to show 
		$$C_{F}\cap[\cup_{u\in L(P_D)}q(u, 1)+C_D]=\cup_{u\in L(F)}q(u, 1)+C_F.$$
		Let $x\in C_{F}\cap[\cup_{u\in L(P_D)}q(u, 1)+C_D]$. Choose $u_0\in L(P_D)$ such that $x=q(u_0,1)+y$ for some
		$y\in C_D$. Since $x\in C_F$, we have $0=\sigma_F(x)=q\sigma_F(u_0,1)+\sigma_F(y)$. This implies $\sigma_F(u_0,1)=\sigma_F(y)=0$, i.e., $(u_0,1), y\in C_F$. Hence $x\in \cup_{u\in L(F)}q(u, 1)+C_F.$ The reverse inclusion follows since $\sigma_F((u, 1)+y)=0$ for all $u\in L(F)$ and $y\in C_F$.
		
		Proof of Part $(2)$: The first equality is obvious. The second equality follows from Part $(1)$. 
	\end{proof}
	
	\begin{rmk}\label{rmkidealdensityfn}
		Let $(R, {\bf m})$ be the  
		homogeneous coordinate ring of dimension $d\geq 3$, associated 
		to the projectively normal toric pair $(X, D)$.
		Let $I$ be a monomial prime ideal of height one, associated to a facet $F$ of the polytope $P_D$ and let $f_{\overline{R}, \overline{\bf m}}$ be \textnormal{\mbox{HK}} density function of the standard graded ring $\overline{R}:=R/I$ with respect to its homogeneous maximal ideal
		$\overline{\bf m}={\bf m}/I$. For $\lambda\in[0, \infty)$ and $q=p^n, n\in\N$, we have $f_{\overline{R}, \overline{\bf m}}(\lambda)=\lim_n f_n({\overline{R}, \overline{\bf m}})(\lambda)$
		\begin{eqnarray*}\label{hkdqtring}= \lim_n\frac{1}{q^{d-2}}\ell\left(\frac{R}{m^{[q]}+I}\right)_{\lfloor q\lambda \rfloor}=\lim_n\frac{1}{q^{d-2}}\#\left[(C_F\setminus \cup_{u\in L(P_D)}q(u,1)+C_D)\cap\Lambda_{\lfloor q\lambda \rfloor}\right]\\
			= \lim_n\frac{1}{q^{d-2}}\#\left[(C_F\setminus \cup_{u\in L(F)}q(u,1)+C_F)\cap\Lambda_{\lfloor q\lambda \rfloor}\right].
		\end{eqnarray*}
	\end{rmk}
	
	\section{The boundary of $\sP_D$ parallel to the facet $C_F$ of the cone $C_D$}
	In this section, we study the set $\partial(\sP_D)$, the set $\partial(\sP_D)\cap\{x\in C_D\mid \sigma_{F}(x)=\mu\}$, where $\mu=\sigma_{F}(u, 1)$ for some $u\in L(P_D\setminus F)$. We also study the coefficient of the Ehrhart quasi-polynomial of certain polytopes lying inside $\partial(\sP_D)$. We set the following notations first: 
	\begin{notations}\rm\label{n1} 
		\begin{enumerate}
			\item For a convex polytope $Q$, let $v(Q) = \{\mbox{vertices of}~Q\}$ and
			$\sF(Q) = \{\mbox{facets of}~Q\}$.
			\item For a convex polytope $Q\subset \R^d$, and for $\lambda\in [0,\infty]$ we set $Q_{\lambda}=Q\cap\{z=\lambda\}$.
			\item For a set $F\subseteq \R^d$, $\partial(F)=$ boundary of $F$ in $\R^d$ and $F^\circ=F\setminus \partial(F)=$ interior of $F$ in $\R^d$.
			\item For a set $F\subseteq \R^d$, $\partial_C(F)=$ boundary of $F$ in $C_D$ in the subspace topology of $C_D$, thinking of $C_D\subseteq \R^d$. 
			\item For a set $F\subseteq \R^d$, we denote $A(F)=$ affine hull of $F$ in $\R^d$, the smallest affine set containing $F$,
			i.e., $A(F)=\{\sum_i^m a_if_i\mid m\in \N, a_i\in \R, f_i\in F, \sum_{i=1}^ma_i=1\}$.
			\item For a set $F\subseteq \R^d$, we say $y\in \textnormal{\mbox{relint}}(F)$, the relative interior of $F$, if there exists $\epsilon>0$ such that 
			$B_d(y, \epsilon)\cap A(F)\subseteq F$. Here $B_d(y, \epsilon)$ denotes the $d$-dimensional ball of radius $\epsilon$ around $y$.
			\item For a facet $F$ of $P_D$, and for $\mu\in\Q_{> 0}$, we set $H_{F, \mu}=\{x\in \R^d\mid\sigma_F(x)=\mu\}$.
			\item Let $\mu_{D, F}:=\{\mu\in \Q_{>0}\mid \sigma_{F}(u, 1)=\mu \text{ for some }u\in L(P_D\setminus F)\}$.
			\item $\partial_{D, F}=\cup_{\mu\in \mu_{D, F}}\partial(\sP_D)\cap H_{F, \mu}$.
			
		\end{enumerate}
		
	\end{notations}
	
	For a toric pair $(X, D)$, a decomposition of $C_D = \cup_{j=1}^sF_j$ was given in \cite[Lemma 4.5]{MT19}, (for
	$d\geq 3$, as $d=2$ corresponds to $(\P^1, \sO_{\P^1}(n))$, for $n\geq 1$, 
	which is easy to handle directly), 
	where $F_j$'s are 
	$d$-dimensional cones such that, each $P_j := F_j\cap {\overline{\sP}_D}$  
	is a convex rational polytope and is a closure of $P_j':= 
	F_j\cap {\sP_D}$. In \cite{MT20}, the boundary of $\sP_D$ was studied and described in terms of
	the facets of $P_j$'s. We recall the decomposition of $C_D$ and few properties
	of $\partial(\sP_D)$ from \cite{MT19} and \cite{MT20} which are relevant for this work. 
	
	The cone $F_j \in \{d\text{-dimensional cones}\}$, which is the  
	closure of a  connected component of
	$C_D\setminus \cup_{iu}H_{iu}$,
	where the hyperplanes $H_{iu}$ are given by 
	$$H_{iu} = \mbox{the affine hull of}~~\{(v_{ik},1), (u,1),
	({\bf 0})\mid v_{ik}\in v(C_{0i}),~~u\in L(P_D)\},$$
	where $C_{0i} \in \{(d-3)~\mbox{dimensional faces of}~P_D\}$ and ${\bf 0}$ is the 
	origin of $\R^d$.
	For $u\in L(P_D)$, let 
	$$P_j'= F_j \cap\cap_{u\in L(P_D)}(C_u)^c = 
	F_j \cap\cap_{u\in L(P_D)}[C_D\setminus C_u],$$ which is a convex set
	\cite[Lemma 4.5]{MT19}
	and  $P_j=\overline{F_j \cap\cap_{u\in L(P_D)}(C_D\setminus C_u)}$  is the
	$d$-dimensional convex rational polytope
	which is the closure of $P_j'$ in $C_D$ (which equals the closure in $\R^d$). 
	
	Therefore  
	$$\mathcal{P}_D=\cup_{j=1}^s P_j'\quad \mbox{and}\quad 
	\overline{\mathcal{P}}_D=\cup_{j=1}^s P_j,$$
	where $P_1, \ldots, P_s$ are distinct polytopes, whose interiors are disjoint. Moreover, facets of each $P_j$
	are transversal to the $z$-hyperplane, i.e., $\dim(\partial(P_j)\cap\{z=\lambda\})<d-1$ for all $\lambda\in\R$ and for all $j$. Note that 
	$$P_j= {\overline{F_j\setminus \cup_{u\in L(P_D)} C_u}} =
	{\overline{\cap_{u\in L(P_D)} F_j
			\setminus C_u}}=P_j'\sqcup \left(\cup_{u\in L(P_D)}\partial_C(C_u)\cap P_j\right)$$
	and $\partial_C(C_u)\cap P_j = 
	\cup_{\{E\mid E\in \sF(C_u),~E\not\subseteq \partial(C_D)\}} E\cap{P_j}$ \cite[Lemma 8]{MT20}.
	Moreover, for any facet $E\in \sF(P_j)$, either $E\subset E_{j_i}$, for some facet $E_{j_i} \in \sF(F_j)$; 
	or $F\subset F_{u_\nu}$, for some facet $F_{u_\nu}\in \sF(C_u)$ and
	$u\in L(P_D)$. In the later case 
	$F = P_j\cap F_{u_\nu} = P_j\cap A(F_{u_\nu})$, where $F_{u_\nu}\not\subseteq \partial(C_D)$ \cite[Lemma 9]{MT20}.
	Finally we record \cite[Lemma 10]{MT20} which gives the explicit description of $\partial(\sP_D)$ as follows:
	\begin{lemma}\label{l7} \begin{enumerate}
			\item$\partial(\sP_D) = \cup_{\{E\in \sF(P_j)\mid E\neq P_i\cap P_j\}} E$.~~\mbox{In particular}
			\item	$\partial(\sP_D)=\bigcup_{\{ E\in \sF(C_D)\}}E\cap {\overline{\sP}_D} \cup 
			\bigcup_{\{E\in \sF(C_u), u\in L(P_D)\}}E\cap {\overline{\sP}_D}$.
		\end{enumerate}
	\end{lemma}
	\begin{lemma}\label{boundaryinterior}
		Suppose $v\in L(P_D)\setminus L(F')$ for some $F'\in\sF(P_D)$. 
		\begin{enumerate}
			\item Then there exists a $d$-dimensional cone
			$F_j$ occurring in the decomposition of $C_D$ such that $(v,1)\in F_j$ and $[(v,1)+C_{F'}]\cap F_j^\circ\neq \emptyset.$
			\item Moreover, $\dim(\partial(P_j)\cap [(v,1)+C_{F'}])=d-1$, i.e., there exists $\tilde{E}\in \sF(P_j)$ such that $\tilde{E}\subset (v,1)+C_{F'}$.
		\end{enumerate}
	\end{lemma}
	\begin{proof}
		{Proof of Part $(1)$}: We choose small $\epsilon >0$ such that $B_d((v,1),\epsilon)\cap F_j=\emptyset$ for all cones $F_j$ in the decomposition of $C_D$ with $(v,1)\notin F_j$. Therefore $[(v,1)+C_{F'}]\cap B_d((v,1),\epsilon)\subseteq \cup_{(v,1)\in F_j}{F_j}$. If 
		$$[(v,1)+C_{F'}]\cap B_d((v,1),\epsilon)\subseteq \cup_{(v,1)\in F_j}\partial{F_j},$$
		then $d-1=\dim([(v,1)+C_{F'}]\cap B_d((v,1),\epsilon))=\dim([(v,1)+C_{F'}]\cap B_d((v,1),\epsilon)\cap\partial{F_{j_0}})$ for some
		$F_{j_0}$ containing $(v,1)$. Hence 
		$$A((v,1)+C_{F'})=A([(v,1)+C_{F'}]\cap B_d((v,1),\epsilon))=A(F'')$$
		for some $F''\in \sF(F_{j_0})$. This is a contradiction since $F''$ passes through origin in $\R^d$, whereas ${\bf 0}\notin A((v,1)+C_{F'})$ since $(v,1)\notin C_{F'}$. Hence $[(v,1)+C_{F'}]\cap F_j^\circ\neq \emptyset$ for some $F_j$ containing $(v,1)$. 
		
		Proof of Part $(2)$: We take a cone $F_j$ such that $(v, 1)\in F_j$ and $[(v,1)+C_{F'}]\cap F_j^\circ\neq \emptyset$. Note that $\dim([(v,1)+C_{F'}]\cap F_j^\circ)=d-1$. By the claim in the proof of \cite[Lemma 9(2)]{MT20}, we have $[(v,1)+C_{F'}]\cap F_j=A((v,1)+C_{F'})\cap F_j$. Hence
		$$[(v,1)+C_{F'}]\cap F_j^\circ\subseteq \textnormal{\mbox{relint}}((v,1)+C_{F'}),$$ and $\dim(\textnormal{\mbox{relint}}((v,1)+C_{F'})\cap F_j^\circ)=d-1$.  Since $(v,1)\in F_j$, this implies for any ball $B_d((v,1),\epsilon)$ around $(v,1)$ of radius $\epsilon>0$, we have 
		\begin{eqnarray}\label{relintdimn}
			\dim(\textnormal{\mbox{relint}}((v,1)+C_{F'})\cap F_j^\circ\cap B_d((v,1),\epsilon))=d-1.
		\end{eqnarray}
		Since $(v,1)\notin (u,1)+C_D$ for all $u\in L(P_D)\setminus \{v\}$, we can take small $\tilde{\epsilon}>0$ such that 
		$B_d((v,1),\tilde{\epsilon})\cap [(u,1)+C_D]=\emptyset$ for all $u\in L(P_D)\setminus \{v\}$. 
		For any 
		$$y\in \textnormal{\mbox{relint}}((v,1)+C_{F'})\cap F_j^\circ\cap B_d((v,1),\tilde{\epsilon}),$$ we may choose $\epsilon_y>0$ small enough such that
		$B_d(y, \epsilon_y)\subseteq F_j^\circ\cap B_d((v,1),\tilde{\epsilon})$
		and $B_d(y, \epsilon_y)\cap A((v,1)+C_{F'})\subseteq \textnormal{\mbox{relint}}((v,1)+C_{F'})$. Therefore,
		$$B_d(y, \epsilon_y)\cap P_j'=B_d(y, \epsilon_y)\setminus \cup_{u\in L(P_D)}(u,1)+C_D=B_d(y, \epsilon_y)\setminus (v,1)+C_D\neq \emptyset.$$
		Hence $\textnormal{\mbox{relint}}((v,1)+C_{F'})\cap F_j^\circ\cap B_d((v,1),\tilde{\epsilon})\subseteq \partial(P'_j)=\partial(P_j)$. From (\ref{relintdimn}), we have $\dim([(v,1)+C_{F'}]\cap \partial(P_j))=d-1$.
	\end{proof}

	\begin{lemma}\label{lemmacore1}
		\begin{enumerate}
			
			\item For $\mu\in \Q_{>0}$, let $$A_{F, \mu}=(\cup_{\stackrel{u\in L(P_D),}{\sigma_F(u, 1)=\mu}}(u,1)+C_F)\setminus (\cup_{\stackrel{v\in L(P_D),} {\sigma_F(v,1)< \mu}}(v,1)+C_D).$$ Then 
			$$A_{F, \mu}\subseteq \partial(\sP_D)\cap H_{F,\mu} \text{ for all } \mu\in\Q_{>0}.$$
			\item $\partial(\sP_D)\cap H_{F,\mu} \subseteq A_{F, \mu}\cup B_{F, \mu} \cup [\partial(\sP_D)\cap\partial(C_D)\cap H_{F,\mu}]$
			where
			$$B_{F, \mu}=\big[\cup_{\stackrel{v\in L(P_D), \sigma_F(v,1)< \mu}{F\neq F'\in\sF(P_D)}}(v,1)+C_{F'}\big]\cap H_{F, \mu}\cap \partial({\sP}_D).$$
		\end{enumerate}
		
	\end{lemma}
	\begin{proof}
		Proof of Part $(1)$:
		Let $x\in A_{F, \mu}$, i.e., $x\in((u,1)+C_F)\setminus (\cup_{\stackrel{v\in L(P_D),} {\sigma_F(v, 1)< \mu}}C_v)$
		for some $u\in L(P_D\setminus F)$ with $\sigma_F(u,1)=\mu$. 
		Choose a small neighbourhood around $x$ of radius $\epsilon$, $B_d(x, \epsilon)\subset \R^d\setminus (\cup_{\stackrel{v\in L(P_D),} {\sigma_F(v,1)< \mu}}C_v).$ Note that if $B_d(x, \epsilon)\cap\{y\in C_D\mid\sigma_F(y)<\mu\}=\emptyset$ then $B_d(x, \epsilon)\cap C_D\subset \{\sigma_F\geq \mu\}$ which is a contradiction since $\sigma_F(x)=\mu>0$ and we have $\tilde{x}\in C_D$ with $\sigma_F(\tilde{x})<\mu$. Hence, $$\emptyset\neq B_d(x, \epsilon)\cap\{y\in C_D\mid\sigma_F(y)<\mu\}\subset C_D\setminus (\cup_{{v\in L(P_D)}}C_v)=\sP_D,$$  which gives $x\in \overline{\sP}_D.$ Since $A_{F, \mu}\cap \sP_D=\emptyset$, this implies $A_{F,\mu}\subseteq \partial(\sP_D)\cap H_{F,\mu}$.
		
		\vspace{.2cm} 
		
		\noindent Proof of Part $(2)$:  It is enough to show that $[\partial(\sP_D)\setminus \partial(C_D)]\cap H_{F, \mu}\subseteq A_{F, \mu}\cup B_{F, \mu}$. Let $x\in[\partial(\sP_D)\setminus \partial(C_D)]\cap H_{F, \mu}$. Then by Lemma $\ref{l7}(2)$, $x\in E$ where $E\in\sF(C_u)$ for some $u\in L(P_D)$. We split the proof in two cases.
		
		\noindent{\bf Case} $(1)$: Suppose $E=(u, 1)+C_F$ for some $u\in L(P_D)$. This implies $\sigma_F(u,1)=\mu$. Suppose $x\notin A_{F, \mu}$. This implies $x\in (v, 1)+C_D$ for some $v\in L(P_D)$ with $\sigma_F(v, 1)<\mu$. Since $x\in \partial(\sP_D)$, we have $x\notin (v, 1)+C_D^\circ,  i.e.,  x\in \cup_{F'\in\sF(P_D)}(v,1)+C_{F'}$. But
		$x\notin (v, 1)+C_F$ since $\sigma_{F}(v,1)<\mu =\sigma_F(x)$. Hence $x\in \big[\cup_{F\neq F'\in\sF(P_D)}(v,1)+C_{F'}\big]\cap H_{F, \mu}\cap \partial(\sP_D)\subseteq B_{F, \mu}$. 
		
		\noindent{\bf Case} $(2)$: Suppose $x\notin (u, 1)+C_F$ for all $u\in L(P_D)$. Then $E=(v, 1)+C_{F'}$ for some $v\in L(P_D)$ and $F'\in \sF(P_D)$ with $F\neq F'$. Since $x\notin (v, 1)+C_F$ we must have $\sigma_F(v, 1)<\mu$. Hence $x\in B_{F, \mu}$. This proves Part $(2)$.
	\end{proof}
	
	\begin{defn}We recall the definition of Ehrhart quasi-polynomial of a convex polytope $P\subset \R^d$. The function
		$i(P, - ):\N\longto \N$ given by 
		$$i(P, n):=\#(nP\cap \Z^d)=\sum_{j=0}^{\dim(P)}C_j(P, n)n^j,$$
		is a quasi-polynomial of degree $\dim(P)$, i.e., the coefficient $C_j(P, n)$ of $n^j$ is periodic in $n$ for all $j=0,\ldots, n$, and $C_{\dim(P)}$
		is not identically zero. Moreover $C_{\dim(P)}=\textnormal{\mbox{rVol}}_{\dim(P)}(P)$ if $A(P)\cap \Z^d\neq \emptyset$.
	\end{defn}
	
	\begin{notations}\rm\label{ndn}
		
		\begin{enumerate}
			\item In the rest of the paper, 
			for a bounded set $Q\subset \R^d$ and for $n, m\in\N$,  
			we define 
			\begin{equation}\label{dn}i(Q, n, m):= 
				\#(nQ\cap \{z = m\}\cap \mathbb{Z}^{d}),
			\end{equation}
			where $z$ is the $d^{th}$ coordinate function on $\R^d$.
			\item 
			Let $v(\sP_D) := \cup_{j=1}^s\pi(v(P_j))$, where 
			$\pi:\R^d \longto \R$ is the projection 
			given by projecting  to the last coordinate $z$ and the set 
			$\pi(v(P_j)) = \{\rho_{j_1}, \ldots, \rho_{j_{m_j}}\}$,
			$\mbox{with}~\rho_{j_1}<\rho_{j_2}<\cdots < \rho_{j_{m_j}}$. 
			\item Let $S = \{m/q\mid q = p^n, ~m, n\in \Z_{\geq 0}\}\setminus v(\sP_D)$.
		\end{enumerate}
		
	\end{notations}
	
	\begin{lemma}\label{lemmacore2}
		Let $\mu\in \Q_{>0}$. There exists a finite set $S_{F}\subset [0,\infty)$ such that for $\lambda\in S\setminus S_F$ and $q=p^n, n\in\N$, such that $q\lambda\in \Z_{\geq 0}$, 
		\begin{enumerate}
			\item there exists a constant $C_{\mu}>0$ $($independent of $\lambda\in S$ and $n\in\N$$)$ such that
			$$i([\partial(\sP_D)\cap H_{F, \mu}]\setminus A_{F, \mu}, q, q\lambda)=c_\mu(\lambda, n)q^{d-3}$$
			for some constant $c_\mu(\lambda, n)$ with $|c_\mu(\lambda, n)|<C_\mu$. 
			\item there exists a constant $C_1>0$ $($independent of $\lambda\in S\setminus S_F$ and $n\in\N$$)$ such that
			$$i(\partial_{D, F}, q, q\lambda)=i(A_{F}, q, q\lambda)+c^{(1)}_\lambda(n)q^{d-3}$$ for some constant $c^{(1)}_{\lambda}(n)$ with $|c^{(1)}_\lambda(n)|< C_1$. Here $A_F=\cup_{\mu \in \mu_{D, F}} A_{F, \mu}$ and $\partial_{D, F}$ is as in Notations $\ref{n1}(9)$.
		\end{enumerate}
	\end{lemma}
	\begin{proof}
		Proof of Part $(1)$: By Lemma $\ref{lemmacore1}(2)$, $$[\partial(\sP_D)\cap H_{F, \mu}]\setminus A_{F, \mu}\subseteq B_{F, \mu}\cup[\partial(\sP_D)\cap \partial(C_D)\cap H_{F, \mu}].$$
		Note that $B_{F, \mu}= \big[\cup_{\stackrel{v\in L(P_D), \sigma_F(v,1)< \mu}{F\neq F'\in\sF(P_D)}}(v,1)+C_{F'}\big]\cap H_{F, \mu}\cap \partial({\sP}_D)$
		\begin{eqnarray}\label{eq3.2}
			=\bigcup_{\stackrel{\substack{ \{E\in \sF(P_j)~\mid~ E\neq P_i\cap P_j\}_j \\ \{F'\in\sF(P_D)~\mid ~F'\neq F\}}}{\{v\in L(P_D)~\mid ~\sigma_F(v, 1)<\mu\}}}E\cap ((v, 1)+C_{F'})\cap H_{F, \mu}.
		\end{eqnarray}
		The second equality follows from the description of $\partial(\sP_D)$ in Lemma $\ref{l7}(1)$. Note that $$\partial(\sP_D)\cap \partial(C_D)\cap H_{F, \mu}=\cup_{\{F'\in\sF(P_D)~\mid ~F'\neq F\}}C_{F'}\cap\partial(\sP_D)\cap H_{F, \mu}.$$
		Again by Lemma $\ref{l7}(1)$,
		\begin{eqnarray}\label{eq3.3}
			\partial(\sP_D)\cap \partial(C_D)\cap H_{F, \mu}=\bigcup_{\stackrel{\{E\in \sF(P_j)~\mid ~ E\neq P_i\cap P_j\}}{\{F'\in\sF(P_D)~\mid ~F'\neq F\}}}E\cap C_{F'}\cap H_{F,\mu}.
		\end{eqnarray}
		For each convex rational polytope $Q$ appearing in the union (in the right hand side) of Equation (\ref{eq3.2}) and Equation (\ref{eq3.3}), we have $\dim(Q)\leq d-2$, since
		the facet $C_{F'}$ is transversal to $H_{F, \mu}$ for all $F\neq F'\in\sF(P_D)$. Write $B_{F, \mu}\cup[\partial(\sP_D)\cap \partial(C_D)\cap H_{F, \mu}]=\cup_{\gamma\in\Gamma} Q_\gamma$ where $\Gamma$ is a finite index set indexing the finitely many rational polytopes appearing in Equation $(\ref{eq3.2})$ and Equation $(\ref{eq3.3})$. Since $\dim(Q_\gamma)\leq d-2$, if $\dim(Q_{\gamma}\cap\{z=\lambda_{\gamma}\})=d-2$ for some $\lambda_{\gamma}\in [0, \infty)$, then $Q_{\gamma}\subset \{z=\lambda_{\gamma}\}$ \cite[Lemma 14(1)]{MT20}. Hence for atmost one $\lambda\in [0, \infty)$, we have $\dim(Q_{\gamma}\cap\{z=\lambda\})=d-2$. Let $S_{F, \mu}$ denote the (finite) set of all such $\lambda$ 's, i.e., $S_{F, \mu}:=\{\lambda\in [0,\infty)\mid \dim(Q_{\gamma}\cap\{z=\lambda\})=d-2 \text{ for some }\gamma\in\Gamma\}$.

		Now
		$$i([\partial(\sP_D)\cap H_{F, \mu}]\setminus A_{F, \mu}, q, q\lambda)\leq i(\cup_{\gamma\in \Gamma}Q_{\gamma}, q, q\lambda) $$
		\begin{equation}\label{e35}=\sum_{\gamma\in\Gamma}i(Q_{\gamma}, q, q\lambda)+\sum_{\alpha\in \Gamma'}\epsilon_{\alpha}i(Q'_{\alpha}, q, q\lambda)\end{equation}
		where $\Gamma'$ is a index set indexing the rational polytopes which are (finite) intersection of rational polytopes from the set $\{Q_\gamma\mid \gamma\in\Gamma\}$ and $\epsilon_{\alpha}\in\{-1,1\}$ depending on $\alpha\in\Gamma'$. 
		By \cite[Lemma 49]{MT20}, for all $\lambda\in S\setminus S_{F,\mu}$ and for all
		$Q=Q_{\gamma_1}\cap\cdots \cap Q_{\gamma_k}$ ($k\geq 1$), where ${\gamma_i}\in \Gamma$, there exists positive constant $C_Q$ (independent of $\lambda$ and $n$) such that
		$$i(Q, q, q\lambda)=i(Q_{\lambda}, q)={c_{Q}}_{\lambda}(n)q^{d-3}$$
		for some constant $c_{Q_{\lambda}}(n)$ with $|c_{Q_{\lambda}}(n)|<C_Q.$
		Hence the assertion in Part (1) follows from Equation (\ref{e35}).
		
		Proof of Part $(2)$: We set $S_F:=\cup_{\mu\in \mu_{D, F}}S_{F, \mu}$. The proof follows immediately from Part (1) since the set $\mu_{D, F}$ is finite.\end{proof}

	\begin{lemma}\label{ehrhartcoeff}
		Let $Q$ be the convex polytope $E\cap H_{F,\mu}$ for $\mu\in \mu_{D,F}$ and $E\in \sF(P_j)$ for some $j\in\{1,\ldots,s\}$.  Let $\lambda\in S$.
		Suppose, $q=p^n$ for some $n\in\N$ such that $q\lambda\in \Z$. Then
		\begin{enumerate}
			\item $C_{d-2}(Q_{\lambda}, q)=\textnormal{\mbox{rVol}}_{d-2}(Q_{\lambda})$.
			\item If $\dim(Q)=d-1$, then for all $j=1,\ldots, d-3$, we have $C_{j}(Q_{\lambda}, q)<\tilde{C}_Q$ for some constant $\tilde{C}_Q$ independent of $\lambda$ and $n$.
			\item \begin{enumerate}
				\item If $\dim(Q)=d-2$ and $Q$ is transversal to the $\{z=0\}$ hyperplane, or
				\item if $\dim(Q)< d-2$,
			\end{enumerate}
			then $i(Q_{\lambda},q)\leq {C}'_Qq^{d-3}$ for some constant ${C}'_Q$ independent of $\lambda$ and $n$.
		\end{enumerate}
	\end{lemma}
	\begin{proof}
		We set $m=q\lambda$. 
		
		Proof of Part (1): We know $\dim(E)=d-1$ and $E$ is transversal to the $\{z=0\}$ hyperplane. Hence $\dim(Q_{\lambda})\leq \dim(E_{\lambda})\leq d-2$. If $\dim(Q_{\lambda})=d-2$, then $\dim(Q_{\lambda})=\dim(E_{\lambda})$ and 
		$$A(qQ_{\lambda})\cap \Z^d=A(qE_{\lambda})\cap \Z^d=A(qE_{\lambda}\cap\{z=m\})\cap \Z^d\neq\emptyset,$$
		by \cite[Lemma 14(3)]{MT20}.
		Therefore, by the proof of Case (a), \cite[Lemma 33(1)]{MT20}, we have $C_{d-2}(Q_{\lambda}, q)=\textnormal{\mbox{rVol}}_{d-2}(Q_{\lambda})$. If $\dim(Q_{\lambda})<d-2$, then by the proof of Case (b), \cite[Lemma 33(1)]{MT20}, we have $C_{d-2}(Q_{\lambda}, q)=0=\textnormal{\mbox{rVol}}_{d-2}(Q_{\lambda})$.

		Proof of Part (2) follows from the proof of Part (a) of \cite[Lemma 33(2)]{MT20}. 
		
		Proof of Part (3)  follows from the proof of Part (b) of \cite[Lemma 33(2)]{MT20}.
	\end{proof}

	\begin{defn}\label{defnT_F}We define the set
		$$T_F=\cup_{\stackrel{\{E\in \sF(P_j)\mid E\neq P_i\cap P_j\}_j}{\mu\in\mu_{D, F}}}\{\lambda\in [0,\infty)\mid \dim(E\cap H_{F,\mu})=\dim((E\cap H_{F,\mu})_{\lambda})=d-2\}.$$
		Note that the set $T_F$ is finite. 
	\end{defn}
	\begin{rmk}\label{setinclusion}
		\textnormal{Recall the set $S_F$ defined in the proof of Lemma $\ref{lemmacore2}$. We remark that $S_F\subseteq v(\sP_D)\cup T_F$. To prove this we first note that $S_F=S_1\cup S_2$, where
			$$S_1=\bigcup_{\stackrel{\substack{\mu\in\mu_{D,F} \\ \{E\in \sF(P_j)~\mid~ E\neq P_i\cap P_j\}_j}}{\{F'\in\sF(P_D)~\mid ~F'\neq F\}}}\{\lambda\in [0,\infty)\mid \dim((E\cap C_{F'}\cap H_{F, \mu})_{\lambda})=d-2\}$$ and
			$$S_2=\bigcup_{\stackrel{\substack{\mu\in\mu_{D,F}, v\in L(P_D) \\ \{E\in \sF(P_j)~\mid~ E\neq P_i\cap P_j\}_j}}{ \{F'\in\sF(P_D)~\mid ~F'\neq F\}}}\{\lambda\in [0, \infty)\mid \dim((E\cap ((v, 1)+C_{F'})\cap H_{F, \mu})_{\lambda})=d-2\}.
			$$}
		\textnormal{Hence, $S_F\subseteq S_{F,1}\cup S_{F, 2}\cup T_F$ where
			$$S_{F,1}=\bigcup_{\stackrel{\substack{\mu\in\mu_{D,F} \\ \{E\in \sF(P_j)~\mid~ E\neq P_i\cap P_j\}_j}}{\{F'\in\sF(P_D)~\mid ~F'\neq F\}}}\{\lambda\in [0,\infty)\mid E\subseteq H_{F, \mu}, \dim((E\cap C_{F'})_{\lambda})=d-2\}$$ and
			$$S_{F,2}=\bigcup_{\stackrel{\substack{\mu\in\mu_{D,F}, v\in L(P_D) \\ \{E\in \sF(P_j)~\mid~ E\neq P_i\cap P_j\}_j}}{ \{F'\in\sF(P_D)~\mid ~F'\neq F\}}}\{\lambda\in [0, \infty)\mid E\subseteq H_{F, \mu},\dim((E\cap [(v, 1)+C_{F'}])_{\lambda})=d-2\}.$$ It is enough to show $S_{F,1}\cup S_{F,2}\subseteq v(\sP_D)$. Suppose $\lambda\in S_{F,1}$,  i.e., $\dim((E\cap C_{F'})_{\lambda})=d-2$ for some
			$F'\in\sF(P_D)$ with $F'\neq F$ and $E\in\sF(P_j)$ such that $E\subseteq H_{F, \mu}$ for some $\mu\in\mu_{D,F}$ and $E\neq P_i\cap P_j$ for all $i\in\{1,\ldots,s\}$. Since $\dim(C_{F'}\cap \sP_D)=d-1$, there exists
			$\tilde{E}\in\sF(P_k)$ for some $k\in\{1,\ldots,s\}$ such that $\tilde{E}\subseteq C_{F'}$. This implies $E\cap \tilde{E}\subseteq \{z=\lambda\}$, hence $\lambda\in v(\sP_D)$. }
		
		\textnormal{Now suppose $\lambda\in S_{F,2}$ and $\dim((E\cap [(v,1)+C_{F'}])_{\lambda})=d-2$ where $F'\in \sF(P_D)$ with $F'\neq F$, $v\in L(P_D)$ and $E\in \sF(P_j)$ for some $j\in\{1,\ldots,s\}$ such that $E\subseteq H_{F,\mu}$ and $E\neq P_i\cap P_j$ for all $i\in\{1,\ldots,s\}$. If $v\in L(F')$, then $\lambda\in S_{F,1}\subseteq v(\sP_D)$. Therefore, we assume $v\in L(P_D)\setminus L(F')$.
			By Lemma $\ref{boundaryinterior}$, there exists $E_1\in \sF(P_k)$ for some $k\in\{1,\ldots,s\}$ such that $E_1\subset (v, 1)+C_{F'}$. Hence $E\cap E_1\subseteq \{z=\lambda\}$, i.e., $\lambda\in v(\sP_D)$.}
	\end{rmk}
	
	\section{$\beta$-density function for $I=p_F$}
	In the rest of the paper, we assume $(X, D)$ is a projectively normal toric pair.
	\begin{lemma}\label{idealgeneration}
		The ideal $I=p_F$ is generated by the set $\{\chi^{(u, 1)}\mid u\in L(P_D\setminus F)\}$.
	\end{lemma}
	\begin{proof}
		For ${\bf x}\in\Z^{d-1}$ and any integer $m\geq 2$ with $({\bf x}, m)\in C_{I}\cap (\Lambda_m)$, it is enough to show there exists $u\in P_D\cap M$
		such that 
		$({\bf x}, m)-(u, 1)\in C_{I}\cap (\Lambda_{m-1})$. 
		Now $({\bf x}, m)=\sum_{u\in L(P_D)} a_u(u, 1)$ for $a_u\in\Z_{\geq 0}$ (since $P_D$ is a normal polytope) and $$1<m=\sum_{u\in L(P_D)} a_u=\sum_{u\in L(P_D\setminus F)}a_u+\sum_{u\in L(F)}a_u.$$
		If $\sum_{u\in L(F)}a_u\geq 1$, then choose $u_0\in L(F)$ such that $a_{u_0}\geq 1$. Since $({\bf x}, m)\in C_I$, we have $\sum_{u\in L(P_D\setminus F)}a_u>0$, hence $({\bf x}, m)-(u_0, 1)\in C_{I}\cap (\Lambda_{m-1})$.
		If $\sum_{u\in F}a_u=0$, then $\sum_{u\in L(P_D\setminus F)}a_u=m>1$. We choose $u_0\in L(P_D\setminus F)$ such that $a_{u_0}\geq 1$. Then $({\bf x}, m)-(u_0, 1)=\sum_{u_0\neq u\in L(P_D\setminus F)}a_u(u, 1)+(a_{u_0}-1)(u_0,1)$ and $\sigma_F(({\bf x}, m)-(u_0, 1))>0$, i.e., $({\bf x}, m)-(u_0, 1)\in C_I\cap(\Lambda_{m-1})$.
	\end{proof}
	
	\begin{defn}\label{defnPsiF}
		For the monomial prime ideal $I=p_F$ of $R$,
		\begin{enumerate}
			\item we define a sequence of functions $\{\psi_n:[0, \infty)\longto\R_{\geq 0}\}_{n\in \N}$ given by 
			$$\psi_n(\lambda)=\frac{1}{q^{d-2}}\ell\left(\frac{m^{[q]}\cap I}{m^{[q]}I}\right)_{\lfloor q\lambda\rfloor}.$$
			\item We define the `{\it small density function}' $\Psi_F:[0, \infty)\longto \R_{\geq 0}$, given by
			$$\Psi_F(\lambda)=\sum_{\mu\in \mu_{D, F}}\textnormal{\mbox{rVol}}_{d-2}(\partial(\sP_D)\cap H_{F,\mu}\cap\{z=\lambda\})=\textnormal{\mbox{rVol}}_{d-2}(\partial_{D, F}\cap\{z=\lambda\}).$$
			Here for $Q = \cup_iQ_i$, a finite union of
			convex rational polytopes $Q_i\subset \R^d$ with $\dim(Q_i)\leq d'$, such that
			$\dim{(Q_i\cap Q_j)}< d'$, for $Q_i\neq Q_j$,
			we define  $\textnormal{\mbox{rVol}}_{d'}Q = \sum_i\textnormal{\mbox{rVol}}_{d'}Q_i$ and
			$\textnormal{\mbox{rVol}}_{d''}Q = 0$, if $d''> d'$. For a detailed discussion of the definition of relative volume, see \cite[Appendix A, Definition $47$]{MT20}.
			
		\end{enumerate}
	\end{defn}
	\begin{rmk}\label{RmkPsi_F}Recall the set $T_F$ described in Definition \ref{defnT_F}.
		Note that for all $\lambda\in [0,\infty)\setminus T_F$,
		$$\Psi_F(\lambda)=\sum_{\mu\in \mu_{D, F}}\sum_{\{E\in \sF(P_j)\mid E\neq P_i\cap P_j\}_j}\textnormal{\mbox{rVol}}_{d-2}(E\cap H_{F,\mu}\cap\{z=\lambda\}).$$ 
	\end{rmk}
	\begin{rmk}\label{PsiCont}Suppose $Q=E\cap H_{F,\mu}$ for some $E\in \sF(P_j)$, $\mu\in \mu_{D,F}$ and suppose the function
		${\psi}_Q:[0, \infty)\longto \R_{\geq 0}$, given by
		$${\psi_Q}(\lambda)=\textnormal{\mbox{rVol}}_{d-2}(Q\cap\{z=\lambda\}).$$
		If $\dim(Q)=d-1$, then $E\subseteq H_{F, \mu}$ and $Q=E$. Therefore
		${\psi}_Q:[0, \infty)\setminus v(\sP_D)\longto \R_{\geq 0}$, given by $\lambda\longto\textnormal{\mbox{rVol}}_{d-2}(Q\cap\{z=\lambda\})$
		is continuous, by \cite[Remark 36]{MT20}. If $\dim(Q)=d-2$ and $Q$ is transversal to the $\{z=0\}$ hyperplane or $\dim(Q)\leq d-3$, then $\dim(Q_{\lambda})\leq d-3$, hence ${\psi}_Q=0$ on $[0, \infty)$. If $\dim(Q)=d-2$ and $\dim(Q\cap\{z=\lambda_0\})=d-2$ for some $\lambda_0\in[0,\infty)$, then $Q\subseteq \{z=\lambda_0\}$. Hence ${\psi}_Q(\lambda_0)=\textnormal{\mbox{rVol}}_{d-2}(Q\cap\{z=\lambda_0\})$ and ${\psi}_Q(\lambda)=0$ for all
		$\lambda\neq \lambda_0$. Hence, by Remark \ref{RmkPsi_F}, the function
		$\Psi_F:[0, \infty)\setminus (v(\sP_D)\cup T_F)\longto \R_{\geq 0}$, given by $\lambda\mapsto\textnormal{\mbox{rVol}}_{d-2}(\partial_{D,F}\cap\{z=\lambda\})$ is continuous. Moreover, $\Psi_F$ is a compactly supported and piecewise polynomial function \cite[Remark 36]{MT20}.
	\end{rmk}
	\begin{lemma}\label{lemconv}
		For all $\lambda\in[0, \infty)$
		and $q=p^n\in\N$ with $\lambda_n:=\lfloor q\lambda\rfloor /q\in S\setminus (v(\sP_D)\cup T_F)$, we have $$\psi_n(\lambda)=\Psi_F(\lambda_n)+c_\lambda(n)/q, \text{ for some constant } c_\lambda(n),$$
		such that $|c_{\lambda}(n)|<C$, where $C$ is a constant independent of $\lambda$ and $n\in\N$.
	\end{lemma}
	\begin{proof}
		For $\lambda\in \R_{\geq 0}$ and $q=p^n$, let $m=\lfloor q\lambda\rfloor$. 
		Note that $$\ell(m^{[q]}\cap I)_m=\#\left[\big((\cup_{u\in L(P_D)}q(u,1)+C_D)\setminus C_F\big)\cap(\Lambda_m)\right]$$
		$$=\#\left[(\cup_{u\in L(P_D)}q(u,1)+C_D)\setminus (C_F\cap [\cup_{u\in L(P_D)}q(u,1)+C_D])\cap(\Lambda_m)\right].$$
		By proof of Lemma \ref{lemma4.1}, we have
		\begin{eqnarray}\label{eq4.4}
			\nonumber&&\ell(m^{[q]}\cap I)_m=\#\big[(\cup_{u\in L(P_D)}q(u,1)+C_D)\setminus (\cup_{u\in L(F)}q(u,1)+C_F)\cap(\Lambda_m)\big]\\
			&&=\#\big[(\cup_{u\in L(P_D \setminus F)}q(u,1)+C_D)\cup(\cup_{u\in L(F)}q(u,1)+[C_D\setminus C_F])\cap(\Lambda_m)\big].
		\end{eqnarray}
		By Lemma \ref{idealgeneration},
		\begin{eqnarray}\label{eq4.5}
			\nonumber\ell(m^{[q]}I)_m&=&\#\big[(\cup_{u\in L(P_D), v\in L(P_D \setminus F)}q(u,1)+(v, 1)+C_D)\cap(\Lambda_m)\big]\\
			&=&\#\big[(\cup_{u\in L(P_D)}q(u,1)+[C_D\setminus C_F])\cap(\Lambda_m)\big].
		\end{eqnarray}
		The last equation follows since, $(X, D)$ is projectively normal, i.e., $P_D$ is a normal polytope.
		
		From Equation (\ref{eq4.4}) and Equation (\ref{eq4.5}), we have
		\begin{eqnarray}\label{eq4.6c}
			\nonumber&&\psi_n(\lambda)=\frac{1}{q^{d-2}}\#\big[(\cup_{u\in L(P_D \setminus F)}q(u,1)+C_F)\setminus (\cup_{u\in L(P_D)}q(u,1)+[C_D\setminus C_F])\cap(\Lambda_m)\big]\\
			\nonumber&=&\frac{1}{q^{d-2}}\#\big[\bigcup_{\mu\in\mu_{D, F}}(\cup_{\stackrel{u\in L(P_D),}{\sigma_F(u,1)=\mu}}q(u,1)+C_F)\setminus (\cup_{\stackrel{v\in L(P_D),} {\sigma_F(v,1)<\mu}}q(v,1)+C_D)\cap(\Lambda_m)\big]\\
			&=&\frac{1}{q^{d-2}}\#[qA_F\cap(\Lambda_m)\big]=\frac{1}{q^{d-2}}i(A_F, q, q\lambda_n), \text{ where } A_F \text{ is as in Lemma }\ref{lemmacore1}.
		\end{eqnarray}
		If $\lambda_n\notin S_F$, by Equation (\ref{eq4.6c}) and Lemma \ref{lemmacore2}(2), we have 
		$$\psi_n(\lambda)=i(\partial_{D,F}, q, q\lambda_n)/q^{d-2}-c^{(1)}_{\lambda_n}(n)/q$$
		such that $|c^{(1)}_{\lambda_n}(n)|<C_1$ for some constant $C_1$, independent of $\lambda$ and $n$. Hence
		
		\begin{eqnarray*}\label{eq4.7}
			&&\psi_n(\lambda)=
			\frac{1}{q^{d-2}}\#[q(\partial_{D,F})\cap(\Lambda_{q\lambda_n})\big]-\frac{c^{(1)}_{\lambda_n}(n)}{q}\\
			&&=\frac{1}{q^{d-2}}\#\big[\cup_{\stackrel{\{E\in \sF(P_j)\mid E\neq P_i\cap P_j\}_j}{\mu\in\mu_{D, F}}}q(E \cap H_{F, \mu}) \cap(\Lambda_{q\lambda_n})\big]-\frac{c^{(1)}_{\lambda_n}(n)}{q}\\
			&&=\frac{1}{q^{d-2}}\sum_{\mu\in\mu_{D,F}}\sum_{\{E\in \sF(P_j)\mid E\neq P_i\cap P_j\}_j}\#[q(E \cap H_{F, \mu}) \cap (\Lambda_{q\lambda_n})] \\
			&&+ \frac{1}{q^{d-2}}\sum_{\alpha\in J}\epsilon_\alpha\#[qK_{\alpha} \cap (\Lambda_{q\lambda_n})]-\frac{c^{(1)}_{\lambda_n}(n)}{q}
		\end{eqnarray*}
		\text{where }$J$ is the index set indexing the (finite) intersections of elements of the set $\cup_j\{E\in \sF(P_j)\mid E\neq P_i\cap P_j\}_j$, further intersecting with $H_{F, \mu}$ for some $\mu\in\mu_{D, F}$, i.e., $K_{\alpha}=E_{\alpha_1}\cap\cdots\cap E_{\alpha_l}\cap H_{F, \mu}$ \text{ for } $E_{\alpha_i}\in \cup_j\{E\in \sF(P_j)\mid E\neq P_i\cap P_j\}$ with $l\geq 2$ and for some $\mu\in\mu_{D, F}$;
		$\epsilon_\alpha\in \{1,-1\},$ \text{ depending on the} $K_{\alpha}\in J$.  Hence $\psi_n(\lambda)$
		\begin{eqnarray}\label{Psi_nDescription}
			\nonumber&=&\frac{1}{q^{d-2}}\sum_{\stackrel{\{E\in \sF(P_j)\mid E\neq P_i\cap P_j\}_j}{\mu\in\mu_{D, F}}}i(E\cap H_{F,\mu}, q, q\lambda_n) + \frac{1}{q^{d-2}}\sum_{\alpha\in J}\epsilon_\alpha i(K_{\alpha}, q, {q\lambda_n})-\frac{c^{(1)}_{\lambda_n}(n)}{q}\\
			&=&\frac{1}{q^{d-2}}\sum_{E, \mu}i((E\cap H_{F,\mu})_{\lambda_n}, q) + \frac{1}{q^{d-2}}\sum_{\alpha\in J}\epsilon_\alpha i((K_{\alpha})_{\lambda_n}, q)-\frac{c^{(1)}_{\lambda_n}(n)}{q}.
		\end{eqnarray}
		From Equation (\ref{Psi_nDescription}) and using Lemma \ref{ehrhartcoeff}, we have for all $\lambda_n\notin T_F$,
		\begin{eqnarray}\label{Psi}
			\nonumber \psi_n(\lambda)&=&\sum_{\mu\in\mu_{D,F}}\sum_{\{E\in\sF(P_j)\mid E\neq P_i\cap P_j\}_j}\textnormal{\mbox{rVol}}_{d-2}(E\cap H_{F,\mu}\cap\{z=\lambda_n\})\\&&+\frac{c^{(2)}_\lambda(n)}{q}+\frac{1}{q^{d-2}}\sum_{\alpha\in J}\epsilon_\alpha i((K_{\alpha})_{\lambda_n}, q)-\frac{c^{(1)}_{\lambda_n}(n)}{q};
		\end{eqnarray}
		for real number $c^{(2)}_{\lambda}(n)$ such that $|{c^{(2)}_{\lambda}}(n)|\leq C_2$ for some constant $C_2$ independent of $\lambda$ and $n$. Note that 
		$\dim(K_{\alpha})\leq d-2$ for all $\alpha\in J$. Suppose $K_{\alpha}=E_{\alpha_1}\cap\cdots\cap E_{\alpha_r}\cap H_{F, \mu}$ for some $\mu\in\mu_{D, F}$ and $E_{\alpha_i}\in \cup_j\{E\in \sF(P_j)\mid E\neq P_i\cap P_j\}$ with $r\geq 2$. If $\dim((K_{\alpha})_\lambda)=d-2$ for some $\lambda\in [0,\infty)$, then $\dim(E_{\alpha_1}\cap E_{\alpha_2})=\dim((E_{\alpha_1}\cap E_{\alpha_2})_{\lambda})=d-2$. Hence $E_{\alpha_1}\cap E_{\alpha_2}\subseteq \{z=\lambda\}$, i.e., $\lambda\in v(\sP_D)$.
		Therefore, for all $\lambda\in[0, \infty)$ and $q=p^n\in\N$ with $\lambda_n\in S\setminus (v(\sP_D)\cup S_{F}\cup T_F)=S\setminus (v(\sP_D)\cup T_F)$ (by Remark \ref{setinclusion}), we have $\psi_n(\lambda)=$
		\begin{eqnarray}\label{Psi}
			\nonumber &&\sum_{\mu\in\mu_{D,F}}\sum_{\{E\in\sF(P_j)\mid E\neq P_i\cap P_j\}_j}\textnormal{\mbox{rVol}}_{d-2}(E\cap H_{F,\mu}\cap\{z=\lambda_n\})+\frac{c^{(3)}_{\lambda}(n)}{q}+\frac{c^{(2)}_\lambda(n)}{q}-\frac{c^{(1)}_{\lambda_n}(n)}{q};
		\end{eqnarray}
		such that $|c^{(3)}_{\lambda}(n)|<C_3$, for some constant $C_3$, independent of $\lambda$ and $n\in \N$ \cite[Lemma 49]{MT20}.  By Remark \ref{RmkPsi_F}, for all $\lambda\in[0, \infty)$, $q=p^n\in\N$ with $\lambda_n\in S\setminus (v(\sP_D)\cup T_F)$, we have 
		\begin{eqnarray*}
			\psi_n(\lambda)&=&\sum_{\mu\in\mu_{D, F}}\textnormal{\mbox{rVol}}_{d-2}(\partial(\sP_D)\cap H_{F,\mu}\cap\{z=\lambda_n\}) +\frac{c^{(3)}_{\lambda}(n)}{q}+\frac{c^{(2)}_{\lambda}(n)}{q}-\frac{c^{(1)}_{\lambda_n}(n)}{q}.
		\end{eqnarray*}
		Hence the lemma.
	\end{proof}
	
	\begin{lemma}\label{descriptionfn}For $\lambda\in [0, \infty)$,
		$$g_n(I, {\bf m})(\lambda)=g_n(R, {\bf m})(\lambda)-f_n(R/I, {\bf m}/I)(\lambda)+\psi_n(\lambda).$$
	\end{lemma}
	\begin{proof}Let $m=\lfloor q\lambda\rfloor$.
		We have \begin{eqnarray}\label{eq4.6}
			\nonumber g_n(I, {\bf m})(\lambda)&=&\frac{1}{q^{d-2}}[\ell(I/{\bf m}^{[q]}I)_m-f_{I, {\bf m}}(m/q)q^{d-1}]\\
			\nonumber&=&\frac{1}{q^{d-2}}[\ell(I/{\bf m}^{[q]}\cap I)_m+\ell\big(({{\bf m}^{[q]}\cap I})/({{\bf m}^{[q]}I})\big)_m-f_{I, {\bf m}}(m/q)q^{d-1}].
		\end{eqnarray}
		Using the additive property of $\mbox{HK}$ density function \cite[Proposition 2.14]{Tri18}, we have $f_{I, {\bf m}}(\lambda)=f_{R, {\bf m}}(\lambda)$ for all $\lambda\in[0, \infty)$. Hence we have
		\begin{eqnarray}\label{eqnew}
			g_n(I, {\bf m})(\lambda)=\frac{1}{q^{d-2}}[\ell(I/{\bf m}^{[q]}\cap I)_m-f_{R, {\bf m}}(m/q)q^{d-1}]+\psi_n(\lambda).
		\end{eqnarray}
		
		Note that $$\ell(I/{\bf m}^{[q]}\cap I)_m=\#\left[\big([C_D\setminus C_F]\setminus \cup_{u\in P_D}q(u, 1)+C_D\big)\cap(\Lambda_m)\right]$$
		\begin{eqnarray}\label{g_n}\nonumber&=&\#\left[\big(C_D\setminus \cup_{u\in L(P_D)}q(u, 1)+C_D\big)\cap(\Lambda_m)\right]-\#\left[\big(C_F\setminus \cup_{u\in L(P_D)}q(u, 1)+C_D\big)\cap(\Lambda_m)\right]\\
			&=&\#(q\sP_D\cap\Lambda_m)-\#\left[\big(C_F\setminus \cup_{u\in L(P_D)}q(u, 1)+C_D\big)\cap(\Lambda_m)\right].
		\end{eqnarray}

		By Equation (\ref{eqnew}), Equation (\ref{g_n}) and Remark \ref{hkdqtring}, it follows that
		$$g_n(I, {\bf m})(\lambda)=\frac{1}{q^{d-2}}\left[\#(q\sP_D\cap\Lambda_m)-f_{R, {\bf m}}(m/q)q^{d-1}\right]-f_n(R/I, {\bf m}/I)(\lambda)+\psi_n(\lambda)$$
		$$=g_n(R, {\bf m})(\lambda)-f_n(R/I, {\bf m}/I)(\lambda)+\psi_n(\lambda).$$
	\end{proof}
	
	\noindent {\it Proof of Theorem \ref{thidealbeta}.}
	By Remark \ref{PsiCont}, the function $\Psi_F$ is compactly supported and is continuous on $[0, \infty)\setminus v(\sP_D)\cup T_F$.
	Therefore, by  \cite[Theorem 1]{Tri18} and Theorem \ref{thbeta}, the function $g_{I, {\bf m}}=g_{R, {\bf m}}+f_{R/I, {\bf m}/I}+\Psi_F$ is also continuous on $[0, \infty)\setminus v(\sP_D)\cup T_F$. We set $v_{\sP_D, F}=v(\sP_D)\cup T_F$.
	
	Following a similar argument as in \cite[Lemma 39(1)]{MT20} and using Lemma \ref{lemconv}, we see that for any compact set $V\subseteq [0, \infty)\setminus v_{\sP_D, F}$, the sequence of functions $\{\psi_n|_V\}$ converges uniformly to the function $\Psi_F$. Hence the sequence of functions $\{g_n^{(I, {\bf m})}|_V\}$ converges uniformly to the function $g_{I, {\bf m}}|_V$
	follows from Lemma \ref{descriptionfn}, \cite[Theorem 1.1]{Tri18} and Theorem \ref{thbeta}. Following a similar argument given in the proof of \cite[Corollary 3]{MT20}, we get 
	$\int_{0}^{\infty}g_{I, {\bf m}}(\lambda)d\lambda=\beta(I, {\bf m})$.
	
	\vspace{.2cm}
	
	\begin{cor}\label{corsum}With the notations as in Theorem \ref{thidealbeta}, for a projectively normal toric pair $(X, D)$, we have
		$$\sum_{\{F\mid F\text{ is facet of }P_D\}}g_{p_F, {\bf m}}(\lambda)=(r-2)g_{R, {\bf m}}(\lambda)$$
		for all $\lambda\in [0,\infty)\setminus v_{\sP_D,F}$, where $r\in \N$ is the number of facets of the polytope $P_D$.
	\end{cor}
	
	\begin{proof} From Theorem \ref{thidealbeta}, for $\lambda\in[ 0,\infty)\setminus v_{\sP_D,F}$, we have
		$$\sum_{F\in\sF(P_D)}g_{p_F, {\bf m}}(\lambda)=(r)g_{R, {\bf m}}(\lambda)-\sum_{F\in\sF(P_D)}f_{R/p_F, {\bf m}/p_F}(\lambda)+\sum_{F\in\sF(P_D)}\Psi_F(\lambda)$$
		where $r$ is the number of facets of $P_D$. Hence 
		$$\sum_{F\in\sF(P_D)}g_{p_F, {\bf m}}(\lambda)=(r)g_{R, {\bf m}}(\lambda)-\textnormal{\mbox{rVol}}_{d-2}(\partial(\sP_D)\cap\partial(C_D)\cap\{z=\lambda\})$$$$+\textnormal{\mbox{rVol}}_{d-2}([\partial(\sP_D)\setminus \partial(C_D)]\cap\{z=\lambda\})
		=(r-2)g_{R, {\bf m}}(\lambda).$$
		The last equation follows from the description of $g_{R, {\bf m}}$ in Theorem \ref{thbeta}.
	\end{proof}

	\begin{defn}
		
		For the ideal $I=p_F$, define another function $\alpha_{I, {\bf m}}:[0, \infty)\longto \R$ by setting $$\alpha_{I, {\bf m}}(\lambda)=g_{I, {\bf m}}(\lambda)-g_{R, {\bf m}}(\lambda), \text{ for }\lambda\in [0, \infty).$$
		Clearly $\int_{0}^{\infty} \alpha_{I, {\bf m}}(\lambda)=\tau_{\bf m}(I)$. Extend this to define a map $$\alpha_{{\bf m}}:\textnormal{Cl}(R)\longto \sL^{1}([0, \infty))~(\text {the space of integrable functions }f:[0,\infty)\longto \R)$$ such that it is a group homomorphism. Thus our result gives explicit description of the map $\tau_{\bf m}:\textnormal{Cl}(R)\longto \R$ defined in \cite[Theorem 1.9]{HMM04}. In the next section we compute these functions for some toric pairs.
		
	\end{defn}
	\section{some examples and properties}
	\begin{ex}
		\label{ep1}
		In this example, we compute the $\beta$-density functions for the toric 
		pair $(X,D)=(\P^1, \sO_{\P^1}(l))$ for all $l\in \N$. 
		The polytope $P_D$ can be taken to be the line segment $[0, l]$ 
		$($up to translation by integral points$)$. The cone $$C_D= \textnormal{Cone}\langle (0,1), (l, 1)\rangle =\{(x, y)\mid 0\leq x\leq ly\}\subset \R^2.$$
		Let $(R, {\bf m})$ be the associated homogeneous coordinate ring and let $I_1$ and $I_2$ be the monomial prime ideals associated to the facets $C_{F_1}=C_D\cap\{x=0\}$ and $C_{F_2}=C_D\cap\{x=ly\}$ of $C_D$, respectively. One has
		
		\begin{eqnarray*}
			g_{R, {\bf m}}(\lambda)=&
			\begin{cases} 
				1& \text{ if $ 0\leq \lambda < 1$,} \\
				-l &  \text{ if $ 1\leq \lambda < 1+\frac{1}{l}$,}\\
				0 & \text{ if $\lambda>1+\frac{1}{l},$} 
			\end{cases}
		\end{eqnarray*}
		and
		\begin{eqnarray*}
			g_{I_1, {\bf m}}(\lambda)=g_{I_2, {\bf m}}(\lambda)=0 \text{ for all } \lambda\geq 0.
		\end{eqnarray*}
	\end{ex}

	\begin{ex}
		\label{ex2}
		In this example we compute the $\beta$-density functions for the toric 
		pair $(X, D)=(\P^2, \sO_{\P^2}(1))$. 
		The polytope $P_D$ is the convex hull of the points $(0,0), (0,1)$ and $(1, 0)$ in $\R^2$ 
		$($up to translation by integral points$)$. The cone $$C_D= \textnormal{Cone}\langle (0,0,1), (0,1,1), (1,0,1)\rangle =\{(x,y,z)\mid  x, y\geq 0, x+y\leq z\}\subset \R^2.$$
		Let $(R, {\bf m})$ be the associated homogeneous coordinate ring and let $J_1, J_2$ and $J_3$ be the monomial prime ideals associated to the facets $C_{F_1}=C_D\cap\{x=0\}$, $C_{F_2}=C_D\cap\{y=0\}$ of $C_D$ and $C_{F_3}=C_D\cap\{x+y=z\}$, respectively. One has
		\begin{eqnarray*}
			g_{J_i, {\bf m}}(\lambda)=&
			\begin{cases} 
				\lambda/2& \text{ if $ 0\leq \lambda < 1$,} \\
				-\lambda+3/2&  \text{ if $ 1\leq \lambda < 2$,}\\ 
				\lambda/2-3/2 & \text{ if $2\leq \lambda<3,$}\\
				0 & \text{ if $\lambda\geq 3,$} 
			\end{cases}
		\end{eqnarray*}
		and
		\begin{eqnarray*}
			g_{R, {\bf m}}(\lambda)=3 g_{J_i, {\bf m}}(\lambda) \text{ for all } \lambda\geq 0.
		\end{eqnarray*}
	\end{ex}

	\begin{rmk}\label{rmktau}By Theorem \ref{thbeta}, for $\lambda\in[0, \infty)$,
		$$g_{R, {\bf m}}(\lambda) = 
		\textnormal{\mbox{rVol}}_{d-2}\frac{\left(\partial({\sP_D})\cap \partial(C_D)
			\cap \{z=\lambda\}\right)}{2}
		- \frac{\textnormal{\mbox{rVol}}_{d-2}\left([\partial({\sP_D})\setminus \partial(C_D)]\cap 
			\{z=\lambda\}\right)}{2}$$
		$$=\frac{1}{2}\sum_{i=1}^rf_{R/p_{F_i}, {\bf m}/p_{F_i}}(\lambda)-\frac{1}{2}\sum_{i=1}^r\Psi_{{F_i}}(\lambda),$$
		where $\{F_i\}_{i=1}^r$ are the facets of the polytope $P_D$.	\end{rmk}
	\noindent In the next example we compute the functions $f_{R/p_{F_i}, {\bf m}/p_{F_i}}$ and $\Psi_{{F_i}}$, which enables us to describe the
	$\beta$-density functions and $\tau$-density functions of the ring and of the monomial prime ideals of height one.
	\begin{ex}
		\label{ex3}
		We compute the $\beta$-density function and the $\tau$-density function of the monomial prime ideals
		of height one for the Hirzebruch surface $X=\F_a$, which is a ruled surface over $\mathbb{P}^1_K$, 
		where $K$ is a field of characteristic $p> 0$. See \cite{Ful93} for a detailed description 
		of the surface as a toric variety. The $T$-Cartier divisors are given by 
		$$D_i=V(v_i),\ i=1,2 ,3 , 4,\ \text{ where }\ v_1=e_1, v_2=e_2, v_3=-e_1+ae_2, v_4=-e_2$$ 
		and 
		$V(v_i)$  denotes the $T$-orbit closure corresponding to the cone generated by $v_i$. 
		We know the Picard group is generated by $\{D_i \mid i=1, 2,3, 4\}$ over $\mathbb{Z}$. 
		One can check $\text{Pic}(X)=\mathbb{Z}D_1\oplus\mathbb{Z}D_4$ and $D=cD_1+dD_4$ is ample if and only if 
		$c, d>0.$ Then 
		$$P_D=\{(x,y)\in M_\mathbb{R}\ |\ x\geq -c, 0\leq y\leq d, x\leq ay\}.$$ The description of the $\textnormal{\mbox{HK}}$-density function and
		the $\beta$-density function of the aasociated homogeneous coordinate ring $(R, {\bf m})$ can be found in \cite{Tri16}, \cite{MT19}
		and \cite{MT20}.
		The facets of the polytope $P_D$ are given by the hyperplanes $x=0$, $y=0$, $x=ay+c$ and $y=d$. We denote them by $F_1, F_2, F_3$ and $F_4$ respectively. By Remark \ref{rmktau}, to compute the $\beta$-density function and the $\tau$-density function, it is enough to
		compute the functions $f_{R/p_{F_i}, {\bf m}/p_{F_i}}$ and the functions $\Psi_{F_i}$ for $i=1,2,3,4$. We draw the cross section of the set $\partial(\sP_D)$ at $z=\lambda$ level for $\lambda\in [0, \infty)$ and use the interpretation of these functions in Remark \ref{rmkidealdensityfn} and Definition \ref{defnPsiF}, respectively for the computation. We have
		
		\begin{eqnarray*}
			f_{R/p_{F_1}, {\bf m}/p_{F_1}}(\lambda)&=&
			\begin{cases} 
				d\lambda & \text{ if $ 0\leq \lambda < 1,$} \\
				d(d+1-d\lambda)  &  \text{ if $ 1\leq \lambda < 1+\frac{1}{d},$}\\
				0 & \text{ if $\lambda\geq 1+\frac{1}{d},$} 
			\end{cases}
		\end{eqnarray*} 
		\begin{eqnarray*}
			f_{R/p_{F_2}, {\bf m}/p_{F_2}}(\lambda)&=&
			\begin{cases} 
				c\lambda & \text{ if $ 0\leq \lambda < 1,$} \\
				c(c+1-c\lambda)  &  \text{ if $ 1\leq \lambda < 1+\frac{1}{c},$}\\
				0 & \text{ if $\lambda\geq 1+\frac{1}{c},$} 
			\end{cases}
		\end{eqnarray*}
		\begin{eqnarray*}
			f_{R/p_{F_3}, {\bf m}/p_{F_3}}(\lambda)&=&
			\begin{cases} 
				ad\lambda & \text{ if $ 0\leq \lambda < 1,$} \\
				ad(d+1-d\lambda)  &  \text{ if $ 1\leq \lambda < 1+\frac{1}{d},$}\\
				0 & \text{ if $\lambda\geq 1+\frac{1}{d}$} 
			\end{cases}
		\end{eqnarray*}
		and
		\begin{eqnarray*}
			f_{R/p_{F_4}, {\bf m}/p_{F_4}}(\lambda)&=&
			\begin{cases} 
				(ad+c)\lambda & \text{ if $ 0\leq \lambda < 1,$} \\
				(ad+c)(1-(ad+c)(\lambda-1))  &  \text{ if $ 1\leq \lambda < 1+\frac{1}{ad+c},$}\\
				0 & \text{ if $\lambda\geq 1+\frac{1}{ad+c}.$} 
			\end{cases}
		\end{eqnarray*}
		
		To compute the functions $\Psi_{{F_i}}$ for $i=1,2,3,4$, we consider two different cases.
		\begin{enumerate}
			\item $c \geq d :$ We have
			\begin{eqnarray*}
				\Psi_{{F_1}}(\lambda)&=&
				\begin{cases} 
					0 & \text{ if $ 0\leq \lambda < 1,$}\\
					(c+\frac{ad}{2})(d+1)d(\lambda-1) &  \text{ if $ 1\leq \lambda < 1+\frac{1}{ad+c},$}\\
					(c+\frac{ad}{2})(d+1)\frac{1}{a}(c+1-c\lambda) &  \text{ if $ 1+\frac{1}{ad+c}\leq \lambda < 1 +\frac{1}{c},$}\\
					0 &  \text{ if $\lambda\geq 1 +\frac{1}{c}$,}\\
				\end{cases}
			\end{eqnarray*}
			\begin{eqnarray*}
				\Psi_{{F_2}}(\lambda)&=&
				\begin{cases} 
					0 & \text{ if $ 0\leq \lambda < 1,$} \\
					(cd+d+\frac{ad^2}{2}+\frac{ad}{2})c(\lambda-1)  &  \text{ if $ 1\leq \lambda < 1+\frac{1}{c},$}\\
					cd\lambda+\frac{ad^2}{2}+\frac{ad}{2} & \text{ if $1+\frac{1}{c}\leq \lambda<1+\frac{1}{d},$}\\
					0 & \text{ if $\lambda\geq 1+\frac{1}{d},$} 
				\end{cases}
			\end{eqnarray*}
			\begin{eqnarray*}
				\Psi_{{F_3}}(\lambda)&=&
				\begin{cases} 
					0 & \text{ if $ 0\leq \lambda < 1,$} \\
					(c+\frac{ad}{2})(d+1)ad(\lambda-1)  &  \text{ if $ 1\leq \lambda < 1+\frac{1}{ad+c},$}\\
					(c+\frac{ad}{2})(d+1)(c+1-c\lambda) &  \text{ if $ 1+\frac{1}{ad+c}\leq\lambda < 1 +\frac{1}{c},$}\\
					0 & \text{ if $\lambda\geq 1+\frac{1}{c}$} 
				\end{cases}
			\end{eqnarray*}
			
			and 
			\begin{eqnarray*}
				\Psi_{{F_4}}(\lambda)&=&
				\begin{cases} 
					0 & \text{ if $ 0\leq \lambda < 1,$} \\
					(cd+d+\frac{ad^2}{2}-\frac{ad}{2})(ad+c)(\lambda-1)  &  \text{ if $ 1\leq \lambda < 1+\frac{1}{ad+c},$}\\
					d(ad+c)(\lambda-1) + c+\frac{ad^2}{2}-\frac{ad}{2}&  \text{ if $ 1+\frac{1}{ad+c}\leq\lambda < 1 +\frac{1}{d},$}\\
					
					0 &  \text{ if $\lambda\geq 1 +\frac{1}{d}.$}\\
				\end{cases}
			\end{eqnarray*}
			
			\vspace{.5cm}
			
			\item $c \leq d :$ We have

			\begin{eqnarray*}
				\Psi_{{F_1}}(\lambda)&=&
				\begin{cases} 
					0 & \text{ if $ 0\leq \lambda < 1,$}\\
					(c+\frac{ad}{2})(d+1)d(\lambda-1) &  \text{ if $ 1\leq \lambda < 1+\frac{1}{ad+c},$}\\
					(c+\frac{ad}{2})(d+1)\frac{1}{a}(c+1-c\lambda) &  \text{ if $ 1+\frac{1}{ad+c}\leq \lambda < 1 +\frac{1}{d},$}\\
					(cd+\frac{ad^2}{2} -\frac{ad}{2})\frac{1}{a}\big(a+1-(ad+c)(\lambda-1)\big) & \text{ if $1+ \frac{1}{d}\leq \lambda < 1 +\frac{a+1}{ad+c},$}\\
					\frac{c}{a}(c+1-c\lambda) &  \text{ if $1 +\frac{a+1}{ad+c}\leq \lambda\leq 1+\frac{1}{c}$,}\\
					0 & \text{ if $\lambda\geq 1+\frac{1}{c}$,}

				\end{cases}
			\end{eqnarray*}
			\begin{eqnarray*}
				\Psi_{{F_2}}(\lambda)&=&
				\begin{cases} 
					0 & \text{ if $ 0\leq \lambda < 1,$} \\
					(cd+d+\frac{ad^2}{2}+\frac{ad}{2})c(\lambda-1)  &  \text{ if $ 1\leq \lambda < 1+\frac{1}{d},$}\\
					0 & \text{ if $\lambda\geq 1+\frac{1}{d},$} 
				\end{cases}
			\end{eqnarray*}
			\begin{eqnarray*}
				\Psi_{{F_3}}(\lambda)&=&
				\begin{cases} 
					0 & \text{ if $ 0\leq \lambda < 1,$} \\
					(c+\frac{ad}{2})(d+1)ad(\lambda-1)  &  \text{ if $ 1\leq \lambda < 1+\frac{1}{ad+c},$}\\
					(c+\frac{ad}{2})(d+1)(c+1-c\lambda) &  \text{ if $ 1+\frac{1}{ad+c}\leq\lambda < 1 +\frac{1}{d},$}\\
					
					(c+\frac{ad}{2})(d-1)\big(a+1-(ad+c)(\lambda-1)\big)\\+
					c(c+1-c\lambda) & \text{ if $1+ \frac{1}{d}\leq \lambda < 1 +\frac{a+1}{ad+c},$}\\
					c(c+1-c\lambda)&  \text{ if $1 +\frac{a+1}{ad+c}\leq \lambda < 1+\frac{1}{c}$,}\\
					0 & \text{ if $\lambda\geq 1+\frac{1}{c}$} 
				\end{cases}
			\end{eqnarray*}
			and 
			\begin{eqnarray*}
				\Psi_{{F_4}}(\lambda)&=&
				\begin{cases} 
					0 & \text{ if $ 0\leq \lambda < 1,$} \\
					(cd+d+\frac{ad^2}{2}-\frac{ad}{2})(ad+c)(\lambda-1)  &  \text{ if $ 1\leq \lambda < 1+\frac{1}{ad+c},$}\\
					d(ad+c)(\lambda-1) + c+\frac{ad^2}{2}-\frac{ad}{2}&  \text{ if $ 1+\frac{1}{ad+c}\leq\lambda < 1 +\frac{1}{d},$}\\
					
					(c+\frac{ad}{2})(d-1)\big(a+1-(ad+c)(\lambda-1)\big)
					& \text{ if $1+ \frac{1}{d}\leq \lambda < 1 +\frac{a+1}{ad+c},$}\\
					0 &  \text{ if $\lambda\geq 1 +\frac{a+1}{ad+c}.$}\\
				\end{cases}
			\end{eqnarray*}
			
		\end{enumerate}
		
	\end{ex}
	
	\begin{defn}Let $R$ be a Noetherian standard graded ring of dimension $d\geq 2$ 
		with homogeneous maximal ideal  ${\bf m}$ and let $M$ be a finitely generated non negatively graded
		$R$-module. Let 
		$\ell(M_n) = \frac{e_0(M, {\bf m})}{(d-1)!}n^{d-1} + 
		\tilde{e}_1(M, {\bf m})n^{d-2}+ \cdots + \tilde{e}_{d-1}(M, {\bf m})$
		be the Hilbert polynomial  of $(M,{\bf m})$. 
		Recall the Hilbert density function $F_M:[0,\infty)
		\longto [0, \infty)$, of $M$  as
		$$F_M(\lambda)= \frac{e_0(M, {\bf m})}{(d-1)!}\lambda^{d-1} = 
		\lim_{n\to\infty}F_n(\lambda) 
		:=\frac{1}{q^{d-1}}\ell(M_{\lfloor q\lambda\rfloor}).$$
		Similarly we define the  second Hilbert 
		density function 
		$G_M:[0,\infty)\longto \R$ as
		$$G_M(\lambda)= \tilde{e}_1(M, {\bf m})\lambda^{d-2} = 
		\lim_{n\to\infty}G_n(\lambda):=\frac{1}{q^{d-2}}
		\left(\ell(M_{\lfloor q\lambda\rfloor})-
		F_M\big(\frac{\lfloor q\lambda\rfloor}{q}\big)\right).$$
	\end{defn}

	\begin{propose}\label{p1}
		Let $(R, {\bf m})$ and $(S, {\bf n})$ be two Noetherian standard 
		graded rings over an
		algebraically closed field $K$ $($of characteristic $p>0$$)$ 
		of dimension $d\geq 2$ and $d'\geq 2$, associated to the 
		toric pairs $(X, D)$ and $(Y, D')$, resply. For the monomial prime ideal
		$p_F\# S$ of $R\# S$, we have, 
		$$G_{p_{F}\#S}-g_{p_F\#S, {\bf m}\#{\bf n}}=(G_{p_F}-g_{p_F, {\bf m}})(F_S-f_{S,{\bf n}} )
		+(G_S-g_{S, {\bf n}})(F_{p_F}-f_{p_F, {\bf m}}).$$
		
	\end{propose}
	\begin{proof}
		The proof follows by a similar argument used to prove \cite[Proposition 44]{MT20}.
	\end{proof}
	\begin{rmk}
		With notations as above, using Proposition \ref{p1}, \cite[Proposition 44]{MT20} and \cite[Remark 43]{MT20}, one gets
		$$\alpha_{p_{F}\#S}=\alpha_{p_F, {\bf m}}(F_S-f_{S,{\bf n}} )
		+(G_R-G_{p_F})f_{p_F, {\bf m}}.$$
		This gives a complete description of the $\beta$-density function and the $\tau$-density function for Segre product of toric pairs.
	\end{rmk}
	\begin{ex}
		Let $\M$ be a $2\times 3$ matrix whose entries are the independent variables $x_1,\cdots, x_6$ and let $T$ be the quotient of the ring $k[x_1,\cdots, x_6]$ by the ideal $I_2(\M)$, generated by $2\times 2$ minors of $\M$. In their paper, Huneke, McDermott and Monsky have referred to this example by K. Watanabe where $\beta(T, {\bf m}_T)=-1/4$ to show that the map $\tau:\textnormal{Cl}(T)\longto \R$ is not necessarily a zero map. In this example we compute the map $\tau:=\tau_{{\bf m}_T}$ for all height one monomial prime ideals. Let $(R, {\bf m}_R)$ and $(S, {\bf m}_S)$ be the homogeneous coordinate ring for the toric pairs $(\P^1, \sO_{\P^1}(1))$ and $(\P^2, \sO_{\P^2}(1))$ respectively. Let $I_1, I_2\subseteq R$ and $J_1,J_2, J_3\subseteq S$ be the monomial prime ideals of height one of $R$ and $S$ respectively. For the monomial prime ideals of height one $I_i\#S$ and $R\#J_j$ of $R\#S$, we compute the $\beta$-density function with respect to the homogeneous maximal ideal ${\bf m}_T={\bf m}_R\#{\bf m}_S$. We have
		\begin{eqnarray*}
			\beta_{R\#S, {\bf m}_T}(\lambda)&=&
			\begin{cases} 
				2\lambda^2 & \text{ if $ 0\leq \lambda < 1,$} \\
				2\lambda^2-12(\lambda-1)^2  &  \text{ if $ 1\leq \lambda < 2,$}\\
				2\lambda^2-\frac{15}{2}\lambda +\frac{9}{2}& \text{ if $2\leq \lambda < 3,$}\\
				0 & \text{ if $\lambda\geq 3,$} 
			\end{cases}
		\end{eqnarray*}
		\begin{eqnarray*}
			\beta_{I_i\#S, {\bf m}_T}(\lambda)&=&
			\begin{cases} 
				\frac{3}{2}\lambda^2 & \text{ if $ 0\leq \lambda < 1,$} \\
				\frac{3}{2}\lambda^2-9(\lambda-1)^2  &  \text{ if $ 1\leq \lambda < 2,$}\\
				\frac{3}{2}\lambda^2-\frac{9}{2}\lambda & \text{ if $2\leq \lambda < 3,$}\\
				0 & \text{ if $\lambda\geq 3,$} 
			\end{cases}
		\end{eqnarray*}
		and 
		\begin{eqnarray*}
			\beta_{R\#J_j, {\bf m}_T}(\lambda)&=&
			\begin{cases} 
				\lambda^2 & \text{ if $ 0\leq \lambda < 1,$} \\
				\lambda^2-6(\lambda-1)^2  &  \text{ if $ 1\leq \lambda < 2,$}\\
				\lambda^2-\frac{9}{2}(\lambda-1) & \text{ if $2\leq \lambda < 3,$}\\
				0 & \text{ if $\lambda\geq 3.$} 
			\end{cases}
		\end{eqnarray*}
		Hence $\tau(I_i\#S, {\bf m}_T)=-\frac{1}{2}$ for $i=1,2$ and $\tau(R\#J_j, {\bf m}_T)=\frac{1}{2}$ for $j=1,2,3$.
	\end{ex}

\end{document}